\documentclass[jsl,reqno]{asl}

\usepackage{savesym}
\usepackage{tikz}
\usetikzlibrary{matrix,arrows}
\usepackage{txfonts}
\usepackage{tikz-cd}
\usepackage{stmaryrd}
\usepackage{graphicx}
\usepackage{scalerel}

\title[Model completions for universal classes of algebras]
{Model completions for universal classes of algebras: necessary and sufficient conditions}

\author{George Metcalfe}
\revauthor{Metcalfe, George}
\address{Mathematical Institute\\
University of Bern\\
Sidlerstrasse 5\\
3012 Bern, Switzerland}
\email{george.metcalfe@math.unibe.ch}

\author{Luca Reggio}
\revauthor{Reggio, Luca}
\address{Department of Computer Science\\
University of Oxford\\
15 Parks Rd\\
Oxford OX1 3QD, UK}
\email{luca.reggio@cs.ox.ac.uk}

\thanks{Supported by Swiss National Science Foundation grant 200021\textunderscore 184693 and the European Union's Horizon 2020 research and innovation programme under the Marie Sk{\l}odowska-Curie grant agreement 837724.}

\newcommand{\N}{\mathbb{N}}
\newcommand{\R}{\mathbb{R}}
\newcommand{\Th}{\mathrm{Th}}

\newcommand{\into}{\hookrightarrow}
\newcommand{\onto}{\twoheadrightarrow}
\renewcommand{\hat}{\widehat}
\newcommand{\trimp}{\rhd}
\newcommand{\trneg}{\nabla}
\newcommand{\tremp}{\Lambda}

\theoremstyle{definition}
\newtheorem{theorem}{Theorem}

\newtheorem{lemma}[theorem]{Lemma}
\newtheorem*{claim*}{Claim}

\newtheorem{corollary}[theorem]{Corollary}
\newtheorem{proposition}[theorem]{Proposition}
\newtheorem{example}[theorem]{Example}
\newtheorem{remark}[theorem]{Remark}

\renewcommand{\labelenumi}{(\roman{enumi})}
\numberwithin{theorem}{section}

\newcommand{\defiff}{\: :\Longleftrightarrow\:}
\newcommand{\lan}{\langle}
\newcommand{\ran}{\rangle}
\newcommand{\e}{{\rm e}}
\newcommand{\zr}{{\rm 0}}
\newcommand{\jn}{\vee}
\newcommand{\mt}{\wedge}
\newcommand{\rd}{\slash}
\newcommand{\ld}{\backslash}
\newcommand{\eq}{\approx}
\newcommand{\eqv}{\equiv}

\newcommand{\m}[1]{{\bf {#1} }}
\newcommand{\cop}[1]{\mathbb{#1}}
\newcommand{\f}{\ensuremath{\varphi}}
\newcommand{\ps}{\ensuremath{\psi}}
\newcommand{\ga}{\ensuremath{\gamma}}
\newcommand{\si}{\ensuremath{\sigma}}
\newcommand{\de}{\ensuremath{\delta}}
\DeclareMathOperator{\Con}{\mathrm{Con}}

\newcommand{\cg}[1]{{\rm Cg}_{{#1}}}

\newcommand{\cls}[1]{\mathsf{#1}}
\newcommand{\mdl}[1]{\models_{#1}}
\newcommand{\nmdl}[1]{\not\models_{#1}}
\newcommand{\De}{\Delta}
\newcommand{\Ga}{\Gamma}
\newcommand{\Si}{\Sigma}

\newcommand{\ep}{\varepsilon}
\newcommand{\Diag}{\mathfrak{D}^+}
\newcommand{\F}{\m{F}}
\newcommand{\Fc}{\mathrm{F}}
\newcommand{\Tm}{\m{Tm}}
\newcommand{\Tmc}{\mathrm{Tm}}
\newcommand{\lang}{\mathcal{L}}
\newcommand{\langext}{\mathcal{L}_\rhd}
\newcommand{\xbar}{\overline{x}}
\newcommand{\ybar}{\overline{y}}
\newcommand{\wbar}{\overline{w}}
\newcommand{\zbar}{\overline{z}}
\newcommand{\abar}{\overline{a}}
\newcommand{\bbar}{\overline{b}}

\newcommand{\V}{\cls{V}}
\newcommand{\Vext}{\cls{V_\rhd}}
\newcommand{\Vtot}{\cls{V^c}}
\newcommand{\Vtotext}{\cls{V^c_\rhd}}
\newcommand{\K}{\cls{K}}
\newcommand{\MV}{\cls{MV}}
\newcommand{\MVext}{\cls{MV_\rhd}}
\newcommand{\HA}{\cls{HA}}
\newcommand{\HAC}{\cls{HA^c}}
\newcommand{\MVC}{\cls{MV^c}}
\newcommand{\LA}{\cls{LA}}
\newcommand{\LAext}{\cls{LA_\rhd}}
\newcommand{\OG}{\cls{LA^c}}
\newcommand{\MVD}{\cls{MV}_{\!\Delta}}

\newcommand{\CPRL}{\cls{CPRL}}
\newcommand{\DL}{\cls{DL}}
\newcommand{\DLC}{\cls{DL^c}}
\newcommand{\AND}{\mathbin{\&}}
\newcommand{\OR}{\curlyvee}
\newcommand{\IMP}{\to}
\newcommand{\BIMP}{\leftrightarrow}
\newcommand{\himp}{\supset} 
\DeclareMathOperator*{\bigAND}{\scalerel*{\AND}{\textstyle\sum}}
\DeclareMathOperator*{\bigOR}{\scalerel*{\OR}{\textstyle\sum}}

\begin{document}

\begin{abstract}
Necessary and sufficient conditions are presented for the (first-order) theory of a universal class of algebraic structures (algebras) to have a model completion, extending a characterization provided by Wheeler. For varieties of algebras that have equationally definable principal congruences and the compact intersection property, these conditions yield a more elegant characterization obtained (in a slightly more restricted setting) by Ghilardi and Zawadowski. Moreover, it is shown that under certain further assumptions on congruence lattices, the existence of a model completion implies that the variety has equationally definable principal congruences. This result is then used to provide necessary and sufficient conditions for the existence of a model completion for theories of Hamiltonian varieties of pointed residuated lattices, a broad family of varieties that includes lattice-ordered abelian groups and MV-algebras. Notably, if the theory of a Hamiltonian variety of pointed residuated lattices has a model completion, it must have equationally definable principal congruences. In particular, the theories of  lattice-ordered abelian groups and MV-algebras do not have a model completion, as first proved by Glass and Pierce, and Lacava, respectively. Finally, it is shown that certain varieties of pointed residuated lattices generated by their linearly ordered members, including lattice-ordered abelian groups and MV-algebras, can be extended with a binary operation to obtain theories that do have a model completion.
\end{abstract}

%%%%%%%%%%%%%%%%%%%%%%%%%%%%%%%%%%%%%%%%%%

\maketitle

%%%%%%%%%%%%%%%%%%%%%%%%%%%%%%%%%%%%%%%%%%

\section{Introduction}

The main aim of this paper is to understand what it means in algebraic terms for the (first-order) theory of a universal class of algebraic structures (algebras) to have a model completion. For classes that have finite presentations --- including all quasivarieties, but not, for example, ordered abelian groups --- a complete characterization was provided by Wheeler in~\cite{Whe76} (see also~\cite{Whe78}) using the well-studied  properties of amalgamation and coherence together with a more complicated property referred to as the conservative congruence extension property. However, as we show in Section~\ref{s:model-completions}, replacing coherence and the conservative congruence extension property with a variable projection property and conservative model extension property yields necessary and sufficient conditions for all universal classes of algebras (Theorem~\ref{t:charact-model-completion}).

Although the mentioned properties can be used to confirm that the theories of ordered abelian groups and linearly ordered MV-algebras have a model completion~\cite{Rob77,LS77}, the conservative model extension property is, in general, rather difficult to prove or refute. We therefore also provide, in Section~\ref{s:model-completion-edpc}, a more elegant characterization for varieties (equational classes) of algebras that have equationally definable principal congruences and the compact intersection property, where the conservative model extension property is replaced by a more amenable equational variable restriction property (Theorem~\ref{t:GZ-charact}). This result generalizes slightly a characterization given by Ghilardi and Zawadowski in~\cite{GZ97} (see also~\cite{GZ02}) by covering varieties such as lattice-ordered abelian groups for which there exists no equation that entails all other equations.

In Section~\ref{s:pdpc}, we prove that for any congruence distributive variety $\V$ that has both the congruence extension property and a ``guarded'' deduction theorem, if the theory of $\V$ has a model completion, then $\V$ has equationally definable principal congruences (Theorem~\ref{t:mc-implies-EDPC-guarded}). Our approach is inspired by the work of Glass and Pierce on lattice-ordered abelian groups in~\cite{GP80} and indeed yields both their result that the theory of this variety does not have a model completion, and the same result  for MV-algebras, first proved by Lacava in~\cite{Lac79}. More generally, in Section~\ref{s:prls}, we use this theorem to show that the theory of a Hamiltonian variety of pointed residuated lattices --- spanning varieties of algebras for substructural logics as well as lattice-ordered abelian groups and MV-algebras (see, e.g.,~\cite{BT03,GJKO07,MPT10}) --- has a model completion if, and only if, the variety is coherent and has equationally definable principal congruences, the amalgamation property, and the equational variable restriction property (Theorem~\ref{t:HamRL-charact}).

Finally, in Section~\ref{s:extension}, we associate with any variety $\V$ generated by a class of linearly ordered pointed residuated lattices, a variety $\Vext$ of algebras with an additional binary operation that has equationally definable principal congruences and the same universal theory as $\V$ in the original language. We then show that if $\V$ satisfies a certain syntactic condition, the theory of $\Vext$ has a model completion (Theorem~\ref{t:two-sided-mc}). Notably, this is the case for lattice-ordered abelian groups and MV-algebras, the second case yielding an alternative proof of the fact that the theory of MV$_\De$-algebras has a model completion, first announced by X.~Caicedo at a conference in 2008.

%%%%%%%%%%%%%%%%%%%%%%%%%%%%%%%%%%%%%%%%%%

\section{Algebraic properties}\label{s:algebraic-properties}

Let us first recall some elementary material on universal algebra, referring to~\cite{BS81} for proofs and references. For convenience, we will assume throughout this paper that $\lang$ is an algebraic language (i.e., a first-order language with no relation symbols) containing at least one constant symbol $c$ and that an {\em $\lang$-algebra} $\m{A}$ is an  $\lang$-structure with universe $A$, calling $\m{A}$ {\em trivial} if $|A| = 1$. We denote by $\Con{\m{A}}$ the congruence lattice of $\m{A}$, and  by $\cg{\m{A}}(S)$ the congruence of $\m{A}$ generated by $S\subseteq A^2$.

The {\em term algebra}  $\Tm_\lang(\xbar)$ for $\lang$ over a set $\xbar$ is an $\lang$-algebra with universe $\Tmc_\lang(\xbar)$ consisting of $\lang$-terms with variables in $\xbar$, such that for any $\lang$-algebra $\m{A}$ and map $f\colon\xbar\to A$, there exists a unique homomorphism $\tilde{f}\colon\Tm_\lang(\xbar)\to\m{A}$ extending $f$. Atomic $\lang$-formulas are {\em $\lang$-equations}, written $s \eq t$, and $\lang$-equations and their negations are {\em $\lang$-literals}. Conjunctions, disjunctions, negations, implications, and bi-implications of $\lang$-formulas are built using the symbols $\AND$, $\OR$, $\neg$, $\IMP$, and $\BIMP$, respectively, with $\top\coloneqq c\eq c$ and $\bot\coloneqq \neg \top$. An {\em $\lang$-quasiequation} is an $\lang$-formula $\pi \IMP \ep$, where $\pi$ is a conjunction of $\lang$-equations and $\ep$ is an $\lang$-equation, assuming that the empty conjunction is $\top$. For an $\lang$-term $t$ or $\lang$-formula $\alpha$, we denote by $t(\xbar)$ or $\alpha(\xbar)$ that its free variables belong to the set $\xbar$, and for a conjunction of $\lang$-literals $\xi$, we write $\xi^+$ and $\xi^-$ for the conjunctions of $\lang$-equations occurring in $\xi$ positively and negatively, respectively.

Let $\cop{H}$, $\cop{I}$, $\cop{S}$, $\cop{P}$, and $\cop{P}_U$ denote the class operators of taking homomorphic images, isomorphic images, subalgebras, products, and ultraproducts, respectively. A class of $\lang$-algebras $\K$ is called a {\em variety} if it is closed under $\cop{H}$, $\cop{S}$, and $\cop{P}$, and a {\em quasivariety} if is closed under $\cop{I}$, $\cop{S}$, $\cop{P}$, and $\cop{P}_U$. The class $\K$ is a variety if, and only if, it is an equational class, and a quasivariety if, and only if, it is a quasiequational class. Moreover, $\K$ is a universal class if, and only if, it is closed under  $\cop{I}$, $\cop{S}$, and $\cop{P}_U$, and a positive universal class if, and only if, it is closed under $\cop{H}$, $\cop{S}$, and $\cop{P}_U$.

The variety and quasivariety generated by a class of $\lang$-algebras $\K$ are, respectively, the smallest variety $\cop{HSP}(\K)$  and quasivariety $\cop{ISPP}_U(\K)$ of $\lang$-algebras containing $\K$. An $\lang$-equation is valid in $\K$ if, and only if, it is valid in $\cop{HSP}(\K)$, and an $\lang$-quasiequation is valid in $\K$ if, and only if, it is valid in $\cop{ISPP}_U(\K)$. For future reference, let us note also that any class of $\lang$-algebras $\K$ that is closed under taking finite products satisfies the following {\em disjunction property}: for any conjunctions of $\lang$-equations $\f,\pi_1,\dots,\pi_n$, 
\[
\K\models \f\IMP\!\!\bigOR_{1\leq i\leq n}\pi_i \iff \K\models \f\IMP \pi_i\,\text{ for some $i\in\{1,\dots,n\}$.}
\]
For a variety of $\lang$-algebras $\V$, the {\em $\V$-free algebra} $\F_\V(\xbar)$ over a set $\xbar$ may be identified with the quotient $\Tm_{\lang}(\xbar) / \theta_\V(\xbar)$, where $\theta_\V(\xbar)\coloneqq \bigcap{\{\theta\in\Con{\Tm_{\lang}(\xbar)} \mid \Tm_{\lang}(\xbar) / \theta\in\V\}}$. For any $\lang$-equation $\ep\in\Tmc_{\lang}(\xbar)^2$, we let $\hat{\ep}$ denote its image under the natural surjection $\Tm_{\lang}(\xbar)^2\onto \F_\V(\xbar)^2$. The following useful lemma shows that the valid $\lang$-quasiequations of $\V$ can be described in terms of congruences of $\V$-free algebras.

\begin{lemma}[{cf.~\cite[Lemma~2]{MMT14}}]\label{l:compact-congr-and-consequence-rel}
For any variety of $\lang$-algebras $\V$, conjunction of $\lang$-equations $\pi(\xbar)$, and $\lang$-equation $\ep(\xbar)$,
\[
\V\models \pi\IMP \ep \iff \hat{\ep}\in\cg{\F_{\V}(\xbar)}(\{\hat{\si}\mid \si \text{ is an equation of } \pi\}).
\]
\end{lemma} 

A variety $\V$ has the {\em congruence extension property} if for all $\m{A}\in \V$, any congruence of a subalgebra of $\m A$ extends to a congruence of $\m A$. The following lemma provides a useful equivalent characterization of this property in terms of congruences of $\V$-free algebras.

\begin{lemma}[{cf.~\cite[Lemma~17]{MMT14}}]\label{l:cepequiv}
A variety $\V$ has the congruence extension property if, and only if, for any $\theta\in\Con{\F_{\V}(\xbar)}$ and $\theta'\in\Con{\F_{\V}(\xbar,\ybar)}$,
\[
(\theta'\jn\cg{\F_\V(\xbar,\ybar)}(\theta))\cap\Fc_{\V}(\xbar)^2
=
(\theta'\cap\Fc_{\V}(\xbar)^2)\jn\theta.
\]
\end{lemma}
 
In what follows, we will omit mention of the language $\lang$, assuming throughout that a class of algebras $\K$ is a class of $\lang$-algebras, and that terms, equations, formulas, etc.\ are defined over this language.

%%%%%%%%%%%%%%%%%%%%%%%%%%%%%%%%%%%%%%%%%%

\subsection{The variable projection property and coherence}

Let us say that a class of algebras $\K$ has the \emph{variable projection property} if for any finite set $\xbar,y$ and conjunction of equations $\f(\xbar,y)$, there exists a quantifier-free formula $\xi(\xbar)$ such that $\K\models\f\IMP\xi$ and for any equation~$\ep(\xbar)$,
\[
\K\models \f\IMP \ep \:\Longrightarrow\:\K\models \xi\IMP\ep.
\]
If $\xi(\xbar)$ is required to be a conjunction of equations for each $\f(\xbar,y)$, we say that  $\K$ has the \emph{equational variable projection property}.

\begin{remark}\label{r:quasiequationalrestriction}
Since $\K$ satisfies the same quasiequations as the quasivariety $\cop{ISPP}_U(\K)$ that it generates, $\K$ has the equational variable projection property if, and only if, $\cop{ISPP}_U(\K)$ has this property. 
\end{remark}

For varieties, the variable projection property is equivalent both to the  equational variable projection property and to the widely studied algebraic property of coherence. A variety $\V$ is said to be \emph{coherent} if every finitely generated subalgebra of a finitely presented member of $\V$ is finitely presented.  

\begin{proposition}\label{p:coherence-eq-conseq}
The following statements are equivalent for any variety~$\V$:
\begin{enumerate}
\item[(1)] $\V$ is coherent.
\item[(2)] $\V$ has the equational variable projection property.
\item[(3)] $\V$ has the variable projection property.
\end{enumerate}
\end{proposition}
\begin{proof}
The equivalence of (1) and (2) is established in~\cite[Theorem 2.3]{KM18}, and (3) is an immediate consequence of (2). To show that (3) implies (2), we fix a finite set $\xbar,y$ and a conjunction of equations $\f(\xbar,y)$, and let $\xi(\xbar)$ be a quantifier-free formula such that $\V\models\f\IMP\xi$ and for any equation $\ep(\xbar)$, we have $\V\models \f\IMP \ep \: \Longrightarrow \: \V\models \xi\IMP\ep$. Since $\f$ is satisfiable in $\V$ (e.g., in any trivial algebra), we can assume without loss of generality that $\xi=\xi_1\OR\cdots \OR\xi_m$, where $\xi_1,\dots,\xi_m$ are conjunctions of literals satisfiable in $\V$. 

Since $\V\models\f\IMP\xi$, also $\V\models\f\IMP\!\bigOR_{1\leq i\leq m}\xi_i^+$. By the disjunction property for varieties, $\V\models \f\IMP \xi^+_i$ for some $i\in\{1,\dots,m\}$. Let $\si_1,\dots,\si_n$ be the equations of $\xi^-_i$ and consider any equation $\ep(\xbar)$ such that $\V\models \f\IMP \ep$. By assumption, $\V\models \xi\IMP\ep$, so $\V\models \xi_i\IMP\ep$. Hence $\V\models \xi^+_i\IMP\ep\OR\!\bigOR_{1\leq j\leq n}{\si_j}$ and, by the disjunction property for varieties, either $\V\models \xi^+_i\IMP\ep$ or $\V\models \xi_i\IMP\bot$. But $\xi_i$ is satisfiable in $\V$, so $\V\models \xi^+_i\IMP\ep$.
\end{proof}

Following this last proposition, we will refer to a variety throughout this paper as coherent whenever it has the (equational) variable projection property.

\begin{remark}\label{r:locally-finite-coherent}
An algebra $\m{A}$ is {\em locally finite} if every finitely generated subalgebra of $\m{A}$ is finite, and a class of algebras $\K$ is locally finite if each $\m{A}\in\K$ is locally finite. Since any finitely presented algebra of a locally finite variety is finite and any finite algebra of a locally finite variety is finitely presented, every locally finite variety is coherent.
\end{remark}

Below we introduce some well-known classes of algebras that will be employed as running examples throughout the paper. These algebras all possess definable binary operations $\mt$ and $\jn$ such that $a \le b \defiff a\mt b = a$ defines a lattice order with binary meets and joins given by $\mt$ and $\jn$, respectively.  

\begin{example}\label{ex:linear-orders-evrp}
{\em Linear orders with endpoints} may be considered as bounded lattices $\lan L,\mt,\jn,0,1\ran$ where the defined lattice order is linear. The class $\DLC$ of linear orders with endpoints is then a positive universal class of algebras that generates the variety $\DL$ of bounded distributive lattices as a quasivariety. Since $\DL$ is locally finite, it is coherent, and, by Remark~\ref{r:quasiequationalrestriction}, the class $\DLC$ has the equational variable projection property. 
\end{example}

\begin{example}\label{ex:Heyting-algebras-and-chains}
A {\em Heyting algebra} is an algebra $\lan H,\mt,\jn,\himp,0,1\ran$ such that the reduct $\lan H,\mt,\jn,0,1 \ran$ is a bounded distributive lattice and $\himp$ is the right residual of $\mt$; that is, $a \le b\himp c$ if, and only if, $a \mt b \le c$ for all $a,b,c\in H$. The fact that the variety $\HA$ of Heyting algebras is coherent is a direct consequence of Pitts' uniform interpolation theorem for intuitionistic propositional logic~\cite{Pit92}. 
\end{example}

\begin{example}\label{ex:ordered-groups-alm-coh}
A {\em lattice-ordered abelian group} is an algebra $\lan L,\mt,\jn,+,-,0 \ran$ such that $\lan L,+,-,0 \ran$ is an abelian group, $\lan L,\mt,\jn \ran$ is a lattice, and $a\leq b$ implies $a+c\leq b+c$ for all $a,b,c\in L$. Lattice-ordered abelian groups form a variety $\LA$ that is coherent (see~\cite{KM18}) and generated as a quasivariety by the positive universal class $\OG$ of \emph{ordered abelian groups}, i.e., the class of linearly ordered members of $\LA$  (see, e.g.,~\cite{AF88}). It follows from Remark~\ref{r:quasiequationalrestriction} that $\OG$ has the equational variable projection property.
\end{example}

\begin{example}\label{ex:mv-algebras}
An {\em MV-algebra} is an algebra $\lan M,\oplus,\lnot,0\ran$ satisfying the equations
\[
\begin{array}{rlcrl}
{\rm (MV1)} &  x\oplus(y\oplus z)\eq(x\oplus y)\oplus z & & {\rm (MV4)} &  \lnot\lnot x\eq x\\
{\rm (MV2)} &  x\oplus y\eq y\oplus x & & {\rm (MV5)}  &  x\oplus\lnot0 \eq\lnot 0\\
{\rm (MV3)} &  x\oplus0\eq x & & {\rm (MV6)} &  \lnot(\lnot x\oplus y)\oplus y\eq\lnot(\lnot y\oplus x)\oplus x.
\end{array}
\]
The variety $\MV$ of MV-algebras is coherent (see~\cite{KM18}), and generated as a quasivariety by the positive universal class $\MVC$ of MV-algebras that are linearly ordered with respect to the defined lattice operations $x\mt y\coloneqq\lnot(\lnot x \oplus \lnot (\lnot x \oplus y))$ and $x\jn y\coloneqq \lnot(\lnot x \oplus y)\oplus y$ (see,~e.g.,~\cite{COM99}). Again, it follows from Remark~\ref{r:quasiequationalrestriction} that $\MVC$ has the equational variable projection property.
\end{example}

Notable varieties that are not coherent include the varieties of lattices, semigroups, groups, and modal algebras (see~\cite{KM18} for proofs and references).

%%%%%%%%%%%%%%%%%%%%%%%%%%%%%%%%%%%%%%%%%%

\subsection{The conservative model extension property}

Let us say that a class of algebras $\K$ has the {\it conservative model extension property} if for any finite set $\xbar,y$ and conjunction of literals $\ps(\xbar,y)$, there exists a quantifier-free formula $\chi(\xbar)$ satisfying 
 \begin{enumerate}
 \item [(i)] $\K\models \ps\IMP\chi$
 \item [(ii)] for any $\m A \in \K$ generated by $\abar\in A^{\xbar}$ such that $\m A\models\chi(\abar)$ and for any equation $\ep(\xbar)$,
 \[
 \K\models\ps^+\IMP\ep \:\Longrightarrow\: \m A\models\ep(\abar),
 \]
 there exist an algebra $\m B\in\K$ extending $\m A$ and $b \in B$ such that $\m B\models \ps(\abar,b)$.
  \end{enumerate}

\begin{remark}\label{r:satisfiable-and-non-empty}
In the previous definition, we may assume without loss of generality that $\ps$ is satisfiable in $\K$. Just observe that if this is not the case, then $\K\models \ps\BIMP  \bot$ and we can let $\chi\coloneqq \bot$. Moreover, if there is precisely one definable constant in $\K$ (which is the case, e.g., for lattice-ordered abelian groups), we can assume that $\xbar$ is non-empty. To see this, suppose that $\xbar=\emptyset$ and let $\chi\coloneqq \top$. Then (i) is clearly satisfied and for (ii), any $\m{A}\in\K$ generated by $\emptyset$ is trivial and satisfies all equations, so, since $\ps$ is satisfiable in $\K$, we can choose an algebra $\m B\in \K$ extending $\m{A}$ and $b\in B$ such that $\m B\models \ps(b)$.
\end{remark}

In the case where $\K$ is a universal class of algebras admitting finite presentations (in particular, any quasivariety), the conservative model extension property is implied by the conservative congruence extension property introduced by Wheeler in~\cite{Whe76} (see Proposition~\ref{p:CMEP-CCEP} of Appendix~\ref{app:wheeler}). The following proposition, proved here for the sake of completeness, is then a direct consequence of~\cite[Corollary~1, p.~319]{Whe76}.

\begin{proposition}\label{p:locally-finite}
Every locally finite variety has the conservative model extension property.
\end{proposition}
\begin{proof}
Let $\V$ be a locally finite variety and consider a finite set $\xbar,y$ and conjunction of literals $\ps(\xbar,y)$. We can assume that $\ps$ is satisfiable in $\V$. Since $\V$ is locally finite, the finitely generated free algebra $\F_\V(\xbar)$ is finite and has finitely many congruences $\theta_1,\dots,\theta_m$. For each $w\in\Fc_\V(\xbar)^2$, choose an equation $\ep_{w}$ such that $\hat{\ep}_w=w$. Now, for each $j\in\{1,\dots,m\}$, let $\ps_j(\xbar)$ be the conjunction of literals in the set $\{\ep_{w}\mid w\in \theta_j\}\cup\{\neg\ep_{w}\mid w\notin \theta_j\}$. Then, for any algebra $\m A\in \V$ generated by a tuple $\abar\in A^{\xbar}$, we have $\m A\models \ps_j(\abar)$ if, and only if, the homomorphism $\F_\V(\xbar)/\theta_j\to\m A$ mapping $\xbar_j$ to $\abar$, where $\xbar_j$ is the image of the tuple $\xbar$ under the composite map $\Tm(\xbar)\onto \F_{\V}(\xbar)\onto \F_\V(\xbar)/\theta_j$, is an isomorphism.

Let $S$ be the set of all $j\in\{1,\dots,m\}$ such that there exist $\m B\in \V$ extending $\F_\V(\xbar)/\theta_j$ and $b\in B$ satisfying $\m B\models \ps(\xbar_j,b)$. Since $\ps$ is satisfiable in $\V$, there exist $\m B\in \V$ and $(\abar,b)\in B^{\xbar,y}$ such that $\m B\models \ps(\abar,b)$. The following claim then implies that $S\neq\emptyset$.
\begin{claim*}
Suppose that $\m B\in \V$ and $\m B\models \ps(\abar,b)$ for some $(\abar,b)\in B^{\xbar,y}$. Then there exist $j\in\{1,\dots,m\}$ and an embedding $\F_\V(\xbar)/\theta_j\into \m B$ mapping $\xbar_j$ to $\abar$, and hence $j\in S$.
\end{claim*}
\begin{proof}[Proof of Claim]
By the homomorphism theorem for universal algebra (see, e.g.,~\cite{BS81}), it suffices to observe that the image of the homomorphism $\F_\V(\xbar)\to \m B$ mapping the equivalence class of $\xbar$ to $\abar$ is isomorphic to $\F_\V(\xbar)/\theta_j$ for some  $j\in\{1,\dots,m\}$.
\end{proof}
We now prove that the quantifier-free formula
\[
\chi(\xbar)\coloneqq \bigOR_{k\in S}{\ps_k}
\]
satisfies conditions (i) and (ii) in the definition of the conservative model extension property. For (i), we must show $\V\models \ps\IMP \chi$. By the Claim, if $\m B\in \V$ and $(\abar,b)\in B^{\xbar,y}$ satisfy $\m B\models \ps(\abar,b)$, then there exist $j\in\{1,\dots,m\}$ and an embedding $\F_\V(\xbar)/\theta_j\into \m B$ mapping $\xbar_j$ to $\abar$. So $\m B\models \ps_j(\abar)$, which implies $\m B\models \chi(\abar)$. For (ii), suppose that $\m A\in \V$ is generated by a tuple $\abar\in A^{\xbar}$ satisfying $\m A\models\chi(\abar)$. Let $k\in S$ be such that $\m A\models \ps_k(\abar)$, and recall that there is an isomorphism $\F_\V(\xbar)/\theta_k\to \m A$ mapping $\xbar_k$ to $\abar$. By the definition of $S$, there exist $\m B\in \V$ extending $\m A$ and $b\in B$ such that $\m B\models \ps(\abar,b)$.
\end{proof}

\begin{example}\label{ex:ordered-groups-cmep}
The class of ordered abelian groups $\OG$ has the conservative model extension property. Consider a finite set $\xbar,y$ and a conjunction of literals $\ps(\xbar,y)$. We first assume that $y$ appears in each literal of $\ps$ and then settle the general case. For any two terms $s,t$, write $s<t$ for the formula $(s\leq t)\AND \neg(s\eq t)$, where $\leq$ is the definable lattice order. In view of Remark~\ref{r:satisfiable-and-non-empty}, we can assume that $\xbar\neq\emptyset$. Moreover, without loss of generality (because the members of $\OG$ are linearly ordered), we can assume that $\ps=\ps_1\OR\cdots\OR\ps_m$ and that each disjunct is a conjunction of formulas of the form
\[
py\leq  t,\: py\geq  t, \: py< t, \: py> t, 
\] 
where each $t$ is a group term with variables in $\xbar=x_1,\dots,x_n$, and $p$ is a \emph{fixed} non-zero natural number (e.g., the least common multiple of the coefficients of $y$ in the conjuncts). Let $t_{j_1}(\xbar),\dots,t_{j_u}(\xbar)$ be the terms appearing in $\ps_j$ and let $\chi(\xbar)$ be the formula 
\begin{align*}
&\bigAND{\{\bigOR_{1\leq j\leq m} t_{j_i}{\vartriangle} t_{j_k}\mid j_i,j_k\in\{j_1,\dots,j_u\}, {\vartriangle}\in\{<,\leq\}, \text{ and } \OG\models\ps\IMP \bigOR_{1\leq j\leq m} t_{j_i}{\vartriangle}t_{j_k}\}} \: \AND \\
&\bigAND{\{\bigOR_{i\in I}\neg(x_i\eq0)\mid I\subseteq \{1,\dots,n\} \text{ and } \: \OG\models\ps\IMP \bigOR_{i\in I}\neg(x_i\eq 0)\}}.
\end{align*}
Clearly, $\OG\models \ps\IMP\chi$, so condition (i) of the conservative model extension property is satisfied. For (ii), consider $\m A\in \OG$ together with a tuple $\abar\in A^{\xbar}$ such that $\m A\models \chi(\abar)$. If $\m A$ is the one-element group and there is no $\m B\in\OG$ satisfying (ii), then $\OG\models \ps\IMP \bigOR_{i=1}^n{\neg (x_i\eq 0)}$. So $\OG\models \chi\IMP \bigOR_{i=1}^n{\neg (x_i\eq 0)}$ by the definition of $\chi$, contradicting the fact that $\m A\models \chi(\abar)$. If $\m A$ is non-trivial, let $\m B$ be the divisible hull of $\m A$ and note that $\m B$ is an infinite member of $\OG$. We claim that there is a $b\in B$ such that $\m B\models \ps(\abar,b)$. If no such $b$ exists, then, for each $1\leq j\leq m$, a pair of inequations $t_{j_i}(\abar){\vartriangle} py$ and $py{\vartriangle} t_{j_k}(\abar)$ of $\ps_j(\abar,y)$ is unsatisfiable, for ${\vartriangle}\in\{<,\leq\}$. We settle the case where all these inequations are of the form $t_{j_i}(\abar)< py$ and $py< t_{j_k}(\abar)$, the other cases being very similar. Since $\m B$ is divisible and every divisible ordered abelian group is densely ordered, we get $t_{j_k}(\abar)\leq t_{j_i}(\abar)$. Moreover, $t_{j_i}<py$ and $py< t_{j_k}$ entail $t_{j_i}<t_{j_k}$, and hence $\OG\models\chi\IMP \bigOR_{1\leq j\leq m}t_{j_i}<t_{j_k}$. But then $\m A\models \chi(\abar)$ implies $t_{j_i}(\abar)<t_{j_k}(\abar)$ for some $1\leq j\leq m$, a contradiction. 

Finally, if we have a conjunction of literals of the form $\ps(\xbar,y)\AND \ps'(\xbar)$, where $\ps'$ is any conjunction of literals, the quantifier-free formula $\chi\AND \ps'$ satisfies the conditions for the conservative model extension property.
\end{example}

\begin{example}\label{ex:linear-orders-cmep}
The positive universal class $\DLC$ of linear orders with endpoints, in the language of bounded lattices, has the conservative model extension property. Consider a finite set $\xbar,y$ and a conjunction of literals $\ps(\xbar,y)$ satisfiable in $\DLC$.  Suppose that $\xbar=\emptyset$. If there is a non-trivial member of $\DLC$ satisfying $\ps$, then the formula $\chi\coloneqq \neg (0\eq 1)$ satisfies the conditions for the conservative model extension property. On the other hand, if $\ps$ is satisfied only by the trivial algebra, we can set $\chi\coloneqq 0\eq 1$.
 
Hence, let $\xbar\neq\emptyset$. We assume that $y$ appears in all literals of $\ps$; the general case then follows by reasoning as in Example~\ref{ex:ordered-groups-cmep}. Since the members of $\DLC$ are linearly ordered, $\ps$ is equivalent to a formula $\ps_1\OR\cdots\OR \ps_m$, where each $\ps_j$ is a conjunction of formulas of the form $y\leq  x$, $y\geq  x$, $ y< x$, or $y> x$ with $x\in\xbar=x_1,\dots,x_n$. 

Assume first that the trivial algebra satisfies $\ps$ and let $\chi(\xbar)$ be the formula 
\begin{align*}
&\bigAND{\{\bigOR_{(i,j)\in J} x_i{\vartriangle}x_j\mid J\subseteq \{1,\dots,n\}^2, {\vartriangle}\in\{<,\leq\}, \text{ and } \DLC\models\ps\IMP \bigOR_{(i,j)\in J}x_i{\vartriangle}x_j\}}.
\end{align*}
Reasoning as in Example~\ref{ex:ordered-groups-cmep}, it is not difficult to see that $\chi$ satisfies the conditions for the conservative model extension property. Just replace the divisible hull of an ordered abelian group with any dense linear order with endpoints that extends the given linear order with endpoints $\m A\in \DLC$ (e.g., there is an embedding of $\m A$ into the lexicographic product of $\m A$ and $[0,1]$ which preserves the bounds).

In the case where $\ps$ is not satisfied in the trivial algebra, we replace $\chi$ by $\chi \AND \neg(0\eq 1)$.
\end{example}

%%%%%%%%%%%%%%%%%%%%%%%%%%%%%%%%%%%%%%%%%%

\section{Model completions}\label{s:model-completions}

Let us first recall some relevant model-theoretic notions, referring to~\cite[Section~3.5]{CK90} for further details. By a {\em theory} we will always mean a  first-order theory, i.e., a set of sentences over some first-order language $\lang$.  We let $\Th(\K)$ denote the theory of a class $\K$ of $\lang$-structures, i.e., the set of $\lang$-sentences that are satisfied by all members of $\K$. Two theories $T$ and $T'$ are called \emph{co-theories} if they entail the same universal sentences. Semantically, $T'$ is a co-theory of $T$ if, and only if, every model of $T$ embeds into a model of $T'$ and vice versa. A theory $T^*$ is \emph{model complete} if every formula is equivalent over $T^*$ to an existential formula; that is, model complete theories are those in which alternations of quantifiers can be eliminated. Semantically, a theory $T^*$ is model complete if, and only if, every embedding between models of $T^*$ is elementary. A theory $T^*$ is a \emph{model companion} of a theory $T$ if it is a model complete co-theory of $T$. A {\em model completion} of a theory $T$ is a model companion $T^*$ of $T$ such that for any model $M$ of $T$, the theory of $T^*$ together with the diagram of $M$ is complete. 

Let us also recall that a class $\K$ of $\lang$-structures has the {\em amalgamation property} if given any $\m{A},\m{B},\m{C}\in\K$ and embeddings $f\colon\m{A}\to\m{B}$ and $g\colon\m{A}\to\m{C}$, there exist $\m{D}\in\K$ and embeddings $h\colon\m{B}\to\m{D}$ and $k\colon\m{C}\to\m{D}$ satisfying $hf = kg$.

Below, we collect some useful facts related to model completions.

\begin{proposition}[{cf.~\cite[Propositions~3.5.13,~3.5.15,~3.5.18, and~3.5.19]{CK90}}]\label{p:modcomp}\
\begin{enumerate}
\item[\rm (a)] A theory has at most one model companion up to logical equivalence.
\item[\rm (b)] If a $\forall\exists$-theory $T$ has a model companion $T^*$, then $T^*$ is logically equivalent to the theory of the existentially closed models of $T$.
\item[\rm (c)] If $T^*$ is a model companion of a theory $T$, then $T^*$ is a model completion of $T$ if, and only if, the class of models of $T$ has the amalgamation property. 
\item[\rm (d)] A theory $T^*$ is a model completion of a universal theory $T$ if, and only if, $T^*$ is a co-theory of $T$ that admits quantifier elimination.
\end{enumerate}
\end{proposition}

Our aim in this section is to prove the following characterization of universal classes of algebras whose theories have a model completion.

\begin{theorem}\label{t:charact-model-completion}
Let $\K$ be a universal class of algebras. The theory of $\K$ has a model completion if, and only if, $\K$ has the amalgamation property, variable projection property, and conservative model extension property.
\end{theorem}

For universal classes of algebras with finite presentations, Theorem~\ref{t:charact-model-completion}  specializes to~\cite[Theorem~5]{Whe76} (see Appendix~\ref{app:wheeler}). Observe also that for  locally finite varieties, Remark~\ref{r:locally-finite-coherent} and Proposition~\ref{p:locally-finite} yield the following simpler characterization. 

\begin{corollary}[{\cite[Corollary~1 p.~319]{Whe76}}]\label{c:varieties-charact-mc}
Let $\V$ be a locally finite variety. Then the theory of $\V$ has a model completion if, and only if, $\V$ has the amalgamation property.
\end{corollary}

%%%%%%%%%%%%%%%%%%%%%%%%%%%%%%%%%%%%%%%%%%

We first settle the `if' part of Theorem~\ref{t:charact-model-completion}.

\begin{proposition}\label{p:model-completion-univ}
Let $\K$ be a universal class of algebras. If $\K$ has the amalgamation property, variable projection property, and conservative model extension property, then the theory of $\K$ has a model completion.
\end{proposition}
\begin{proof}
Fix a countably infinite set of variables $\zbar$ and let 
\[
J \coloneqq \{(\ps,y) \mid  \ps \text{ is a conjunction of literals with variables in $\zbar$, and } y \in \zbar\}.
\] 
Consider $j = (\ps,y) \in J$. Let $\xbar$ be the set of variables occurring in $\ps$ different from $y$. Since $\K$ has the variable projection property, there exists a quantifier-free formula $\xi_j(\xbar)$ such that $\K\models \ps^+\IMP \xi_j$ and for every equation $\ep(\xbar)$,
\begin{equation}\label{eq:restriction-of-cg-BC}
\K\models \ps^+\IMP \ep \:\Longrightarrow\: \K\models \xi_j\IMP \ep.
\end{equation}
Since $\K$ has the conservative model extension property, there also exists a quantifier-free formula $\chi_j(\xbar)$ satisfying the following conditions:
 \begin{enumerate}
 \item[\rm (i)] $\K\models\ps\IMP\chi_j$
 \item [(ii)] for any  $\m A \in \K$ generated by $\abar\in A^{\xbar}$ such that $\m A\models\chi_j(\abar)$ and for any equation $\ep(\xbar)$,
 \[
 \K\models\ps^+\IMP\ep \:\Longrightarrow\: \m A\models\ep(\abar),
 \]
 there exist an algebra $\m B\in\K$ extending $\m A$ and $b \in B$ such that $\m B\models \ps(\abar,b)$.
\end{enumerate}
We define the first-order sentence 
\[
\tau_j\coloneqq\forall\xbar \, [(\xi_j\AND \chi_j) \IMP \exists y.\ps].
\]
Let $T\coloneqq \Th(\K)$ be the theory of $\K$ and set $T^* \coloneqq T \cup \{\tau_j \mid j \in J\}$. We claim that $T^*$ is a model completion of $T$. In view of Proposition~\ref{p:modcomp}.(d), it suffices to show that $T^*$ has quantifier elimination and is a co-theory of $T$.

\medskip
\noindent{\em $T^*$ has quantifier elimination.}
We prove that for any $ (\ps,y) \in J$,
\begin{equation}\label{eq:quantelim-positive-univ}
T \vdash \forall \xbar \, [(\exists y.\ps )\IMP (\xi_j\AND \chi_j)].
\end{equation}
It  follows then from the definition of $T^*$ and~\eqref{eq:quantelim-positive-univ} that $T^*$ entails that any formula $\exists y.\ps$, where $\ps$ is a conjunction of literals with variables in $\zbar$ and $y\in\zbar$, is equivalent to a quantifier-free formula. So $T^*$ has quantifier elimination (see, e.g.,~\cite[Lemma 3.2.4]{TZ2012}).

For the proof of~\eqref{eq:quantelim-positive-univ}, fix an arbitrary $j = (\ps,y) \in J$. It suffices to show that for any algebra $\m A\in\K$ and map $g \colon \xbar \to A$,
\[
\m A, g \models \exists y. \ps \:\Longrightarrow\: \m A, g \models \xi_j\AND \chi_j.
\]
Suppose that $\m A, g \models \exists y. \ps$. Then $\m{A}, f \models \ps$ for some map $f \colon \xbar,y \to A$ extending $g$. Moreover, since $\K\models\ps^+\IMP \xi_j$ and $\K\models\ps\IMP\chi_j$, it follows that $\m A, f \models \xi_j$ and $\m A, f\models \chi_j$. But $\xi_j$ and $\chi_j$ have variables in $\xbar$, so $\m A, g\models \xi_j$ and  $\m A, g\models \chi_j$, yielding $\m A, g \models \xi_j\AND \chi_j$.

\medskip
\noindent{\em $T^*$ is a co-theory of $T$.}
Since $T \subseteq T^*$, it suffices to show that any universal sentence entailed by $T^*$ is entailed by $T$. First we show that for any $j = (\ps,y) \in J$, algebra $\m A\in\K$, and map $g \colon \xbar \to A$, 
\begin{multline}\label{eq:cotheory-univ}
\m A, g \models \xi_j\AND \chi_j \:\Longrightarrow\: \text{there exist }  \m B\in \K \text{ and }\iota \colon \m A \into \m B \text{ such that } \m B, \iota g \models \exists y. \ps.
\end{multline}
Suppose that $\m A, g \models \xi_j\AND \chi_j$. We will assume first that the unique homomorphism $\tilde{g} \colon \Tm(\xbar) \to \m A$ extending $g$ is surjective, and hence that $\m A$ is generated by the image $\abar$ of $\xbar$ under $g$. Note that $\m A \models \chi_j(\abar)$.  Moreover, if $\ep(\xbar)$ is any equation such that $\K\models\ps^+\IMP\ep$, then \eqref{eq:restriction-of-cg-BC} yields $\K\models\xi_j\IMP\ep$ and, since $\m A \models \xi_j(\abar)$, also $\m A\models \ep(\abar)$. Hence, by (ii), there exist an algebra $\m B$ in $\K$ extending $\m A$ and $b\in B$ such that $\m B \models \ps(\abar,b)$. So $\m B, \iota g \models \exists y. \ps$, where $\iota \colon \m A\into \m B$ is the inclusion map. 

For the general case of~\eqref{eq:cotheory-univ}, let $\m A'$ be the image of $\Tm(\xbar)$ under $\tilde{g}$ in $\m A$. Since $\K$ is a universal class and $\m A'$ embeds into $\m A$, also $\m A'\in \K$. By the previous argument, there exist $\m B'\in \K$ and  $\iota' \colon \m A' \into \m B'$ such that $\m B', \iota'g \models \exists y. \ps$, witnessed by $b \in B'$, say. Since $\K$ has the amalgamation property, we obtain for the injection $\iota'$ and the inclusion $\m A' \into \m A$, an extension $\iota \colon \m A \into \m B$ with $\m B\in \K$ and an embedding $\lambda \colon \m B' \into \m B$ such that the following diagram commutes:
\[\begin{tikzcd}
{\m A'} \arrow[hookrightarrow,xshift=-1pt]{d} \arrow[hookrightarrow]{r}{\iota'} & {\m B'} \arrow[hookrightarrow,xshift=-1pt]{d}{\lambda} \\
{\m A} \arrow[hookrightarrow]{r}{\iota} & {\m B}
\end{tikzcd}\]
Since $\ps$ is quantifier-free, we have $\m B, \iota g \models \exists y. \ps$ witnessed by $\lambda(b)$.

To conclude the proof, let $\alpha$ be a universal sentence such that $T^* \vdash \alpha$. By the compactness theorem of first-order logic, there exists a finite subset $F \subseteq J$ such that $T \cup \{\tau_j \mid j \in F\} \vdash \alpha$. We prove that, in fact, $T \vdash \alpha$. Consider any $\m A\in\K$. Let $\wbar$ be the set of variables appearing in the scope of the universal quantifier in one of the sentences $\{\tau_j \mid j \in F\}$ and let $g \colon \wbar \to A$ be any map. By repeatedly applying~\eqref{eq:cotheory-univ}, we obtain $\m B\in\K$ and $\iota \colon \m A \into \m B$ such that $\m B, \iota g \models \tau_j$ for each $j \in F$. Now, since $T \cup \{\tau_j \mid j \in F\} \vdash \alpha$, we have $\m B, \iota g \models \alpha$. But $\m A$ is a subalgebra of $\m B$ and $\alpha$ is universal, so $\m A, g\models \alpha$. Hence $T \vdash \alpha$ as required.
\end{proof}

\begin{example}\label{ex:ordered-groups-mc}
The positive universal class $\OG$ of ordered abelian groups, defined over the algebraic language with operation symbols $\mt,\jn,+,-,0$, has the amalgamation property~\cite{Pie72}, variable projection property (see Example~\ref{ex:ordered-groups-alm-coh}), and conservative model extension property (see Example~\ref{ex:ordered-groups-cmep}). Hence Proposition~\ref{p:model-completion-univ} yields the well-known fact that the theory of $\OG$ has a model completion~\cite{Rob77}.\footnote{Note that although the class of ordered abelian groups is often defined using other first-order languages (e.g., with $+$ and the relation $\le$), this difference is immaterial for the existence of a model completion.} Similarly, the class $\MVC$ of linearly ordered MV-algebras satisfies the conditions of  Proposition~\ref{p:model-completion-univ} and its theory therefore has a model completion~\cite{LS77}.
\end{example}

\begin{example}
The positive universal class $\DLC$ of linear orders with endpoints, defined over the algebraic language of bounded lattices, has the amalgamation property, variable projection property (see Example~\ref{ex:linear-orders-evrp}), and conservative model extension property (see Example~\ref{ex:linear-orders-cmep}). Hence Proposition~\ref{p:model-completion-univ} yields the well-known fact that the theory of $\DLC$ has a model completion (see,~e.g.,~\cite{CK90}). It also follows easily that the theory of the positive universal class $\HAC$ of linearly ordered Heyting algebras has a model completion. Note that any linear order with endpoints $\m A\in \DLC$ can be expanded to a linearly ordered Heyting algebra by defining $a \himp b \coloneqq 1$ if $a \le b$ and  $a \himp b \coloneqq b$ otherwise. Indeed, the theory of $\HAC$ is a definitional extension of the theory of $\DLC$ (in the sense of, e.g.,~\cite{Hod97}), where the binary function symbol $\himp$ is defined by the formula 
\[
\xi(x,y,z)\coloneqq (x\leq y \AND z\eq 1) \, \OR \, (y<x \AND z\eq y).
\]
Using the fact that $\xi$ is quantifier-free, the existence of a model completion for the theory of $\HAC$ follows directly from the existence of a model completion for the theory of $\DLC$. 
\end{example}

%%%%%%%%%%%%%%%%%%%%%%%%%%%%%%%%%%%%%%%%%%

We now prove the `only if' part of Theorem~\ref{t:charact-model-completion}. Note first that if $\K$ is a universal class of algebras whose theory has a model completion, then $\K$ has the amalgamation property by Proposition~\ref{p:modcomp}.(c). Next, we consider the variable projection property.

\begin{proposition}\label{p:almost-coh-nec}
Let $\K$ be a universal class of algebras. If the theory of $\K$ has a model completion, then $\K$ has the variable projection property.
\end{proposition}
\begin{proof}
Let $T\coloneqq \Th(\K)$ be the theory of $\K$, and let $T^*$ be a model completion of $T$. Fix a finite set $\xbar,y$ and conjunction of equations $\f(\xbar,y)$. Since $T^*$ has quantifier elimination, there exists a quantifier-free formula $\xi(\xbar)$ such that
\begin{equation}\label{eq:quant-free-xi-y}
T^*\vdash \forall \xbar \, [(\exists y.\f)\BIMP \xi].
\end{equation}
It follows from~\eqref{eq:quant-free-xi-y} that $T^*\vdash \forall \xbar,y \, (\f\IMP\xi)$ and, since $T$ and $T^*$ have the same universal consequences, $T\vdash \forall \xbar,y \, (\f\IMP\xi)$. That is, $\K\models\f\IMP\xi$. It remains to show that for any equation $\ep(\xbar)$ satisfying $\K\models\f\IMP\ep$, also $\K\models\xi\IMP\ep$.\footnote{As pointed out by the referee, this implication is satisfied in fact for any quantifier-free formula $\ep(\xbar)$.}

Suppose that $\K\not\models\xi\IMP\ep$. Then there is an algebra $\m A$ in $\K$ and a tuple $\abar\in A^{\xbar}$ such that $\m A\models \xi(\abar)$ and $\m A\not\models \ep(\abar)$. Since $T$ and $T^*$ are co-theories and $T$ is universal, there exists a model $\m B$ of $T^*$ extending $\m A$ such that $\m B\in\K$. Moreover,  $\m A\models \xi(\abar)$ implies $\m B\models \xi(\abar)$, and hence $\m B\models \exists y. \f(\abar,y)$ by~\eqref{eq:quant-free-xi-y}. Pick $b\in B$ such that $\m B\models \f(\abar,b)$. Since $\m A\not\models \ep(\abar)$, we get $\m B\not\models \ep(\abar)$. Hence $\K\not\models\f\IMP\ep$ as required.
\end{proof}

To complete the proof of Theorem~\ref{t:charact-model-completion}, it remains to prove that whenever the theory of a universal class of algebras has a model completion, this class has the conservative model extension property. In fact, we will show below that the existence of a model completion is equivalent to a stronger property that directly implies the conservative model extension property. This stronger property is difficult to check in concrete cases, but will be useful in Section~\ref{s:pdpc} for establishing consequences of the existence of a model completion for the definability of principal congruences.

\begin{proposition}\label{p:strong-CMEP-iff}
Let $\K$ be a universal class of algebras. The theory of $\K$ has a model completion if, and only if, for any finite sets $\xbar,\ybar$ and conjunction of literals $\ps(\xbar,\ybar)$, there is a quantifier-free formula $\chi(\xbar)$ satisfying the following conditions:
 \begin{enumerate}
 \item [(i)] $\K\models \ps\IMP\chi$
  \item [(ii)] for any $\m A\in\K$ and tuple $\abar\in A^{\xbar}$ satisfying $\m A\models\chi(\abar)$, there exist $\m B\in\K$ extending $\m A$ and $\bbar \in B^{\ybar}$ such that $\m B\models \ps(\abar,\bbar)$.
 \end{enumerate}
In particular, if the theory of $\K$ has a model completion, then $\K$ has the conservative model extension property.
\end{proposition}

\begin{proof}
Let $T\coloneqq \Th(\K)$ be the theory of $\K$, and suppose that $T^*$ is a model completion of~$T$. Fix finite sets $\xbar,\ybar$ and a conjunction of literals $\ps(\xbar,\ybar)$. Since $T^*$ admits quantifier elimination, there exists a quantifier-free formula $\chi(\xbar)$ such that
\begin{equation}\label{eq:quant-free-chi-y}
T^*\vdash \forall \xbar \, [(\exists \ybar.\ps)\BIMP \chi].
\end{equation}
It follows that $T^*\vdash \forall \xbar,\ybar \, (\ps\IMP \chi)$, and since $T$ and $T^*$ are co-theories, 
\begin{equation}\label{eq:univ-consequence}
T\vdash \forall \xbar,\ybar \, (\ps\IMP \chi).
\end{equation}
We show now that $\chi$ satisfies (i) and (ii).

Condition (i) is an immediate consequence of~\eqref{eq:univ-consequence}. For (ii), fix an algebra $\m A\in\K$ and a tuple $\abar\in A^{\xbar}$ such that $\m A\models\chi(\abar)$. Since $T$ and $T^*$ are co-theories and $T$ is universal, there exists a model $\m B \in\K$ of $T^*$ that extends $\m A$. Since $\m B\models \chi(\abar)$, it follows by~\eqref{eq:quant-free-chi-y} that there is a tuple $\bbar\in B^{\ybar}$ such that $\m B\models \ps(\abar,\bbar)$. This settles (ii).

For the converse direction, we only sketch the proof, as it is an easy modification of the proof of Proposition~\ref{p:model-completion-univ}. As before, we consider the theory $T^* \coloneqq T \cup \{\tau_j \mid j \in J\}$, but in this case $\tau_j$ is defined as $\forall\xbar \, (\chi_j \IMP \exists y.\ps)$, where the quantifier-free formula $\chi_j$ satisfies (i) and (ii). To show that $T^*$ has quantifier elimination, it suffices to show that  $T \vdash \forall \xbar \, (\exists y.\ps \IMP \chi_j)$ for all $j=(\ps,y) \in J$, which follows by (i) (cf.~the proof of Proposition~\ref{p:model-completion-univ}). On the other hand, (ii) yields the following property similar to~\eqref{eq:cotheory-univ}: if $\m A, g \models \chi_j$, then there exist $\m B\in \K$ and $\iota \colon \m A \into \m B$ such that $\m B, \iota g \models \exists y. \ps$. Reasoning as in the last part of the proof of Proposition~\ref{p:model-completion-univ}, we conclude that $T$ and $T^*$ are co-theories. It follows then by Proposition~\ref{p:modcomp}.(d) that $T^*$ is a model completion of $T$.
\end{proof}

\begin{remark}
The first part of Proposition~\ref{p:strong-CMEP-iff} admits the following reformulation. Let $\K$ be a universal class of algebras and let $T\coloneqq \Th(\K)$ be its first-order theory. Then the following statements are equivalent:
\begin{enumerate}
\item[(1)] $T$ has a model completion.
\item[(2)] For any finite sets $\xbar,\ybar$ and conjunction of literals $\ps(\xbar,\ybar)$, there is a quantifier-free formula $\chi(\xbar)$ with the same universal consequences, with respect to $T$, as $\exists\ybar.\psi$.
\end{enumerate}
This equivalence was established (for arbitrary first-order languages) by Millar in~\cite[Theorem~3.1]{Millar1995} (see also~\cite[Theorem~3.5.20]{CK90} for a related result) and has recently been considered in the context of the verification of data-aware processes~\cite[Theorem~3.2]{GhilardiEtAl2021}.
\end{remark}

%%%%%%%%%%%%%%%%%%%%%%%%%%%%%%%%%%%%%%%%%%

\section{Compact congruences}\label{s:model-completion-edpc}

In this section, we show that for varieties of algebras satisfying certain congruence lattice conditions, the rather complicated conservative model extension property appearing in Theorem~\ref{t:charact-model-completion} can be replaced by a more amenable equational variable restriction property. The resulting characterization is a slightly more general version of a theorem proved by Ghilardi and Zawadowski in~\cite[Theorem~4]{GZ97}.

Recall that the compact (equivalently, finitely generated) congruences of an algebra $\m{A}$ ordered by set-theoretic inclusion always form a join-semilattice. A class of algebras $\K$ is said to have the \emph{compact intersection property} if the join-semilattice of compact congruences of each $\m{A}\in\K$ forms a lattice, i.e., the intersection of any two compact congruences of $\m{A}$ is compact. Recall also that $\K$ is said to be \emph{congruence distributive} if the congruence lattice of each $\m A\in\K$ is distributive. The following lemma describes a useful syntactic consequence of these two properties.

\begin{lemma}\label{l:consequence-of-CIP}
Let $\V$ be a congruence distributive variety that has the compact intersection property. For any finite set $\xbar$ and conjunctions of equations $\f_1(\xbar),\f_2(\xbar)$, there exists a conjunction of equations $\pi(\xbar)$ such that for any equation $\ep(\xbar,\ybar)$,
\[
\V\models (\f_1\OR \f_2)\IMP\ep\iff\V\models \pi\IMP \ep.
\]
\end{lemma}
\begin{proof}
Let $\xbar$ be any finite set and let $\f_1(\xbar),\f_2(\xbar)$ be conjunctions of equations. We fix $\zbar$ to be any countably infinite set with $\xbar\cap\zbar=\emptyset$ and define $\theta_i\coloneqq \cg{\F_\V(\xbar,\zbar)}(\{\hat{\si}\mid \si \text{ is an equation of } \f_i\})$ for $i\in\{1,2\}$. Since $\theta_1$ and $\theta_2$ are compact, so, by assumption, is their intersection; that is, there exist a finite set $\wbar\subseteq\zbar$ and equations $\si_1(\xbar,\wbar),\dots,\si_n(\xbar,\wbar)$ such that $\{\hat{\si}_1,\dots,\hat{\si}_n\}$ generates $\theta_1\cap \theta_2$. Let $f\colon\xbar,\zbar\to\Fc_\V(\xbar,\zbar)$ be a map sending each $x\in\xbar$ to $x$ and each $w\in\wbar$ to some $u\in\Fc_\V(\xbar)$, such that the image of $f$ contains $\xbar,\zbar$. Clearly, the unique homomorphism $g\colon\F_\V(\xbar,\zbar)\to\F_\V(\xbar,\zbar)$ extending $f$ is surjective. Define $g^*\colon\Con{\F_\V(\xbar,\zbar)}\to\Con{\F_\V(\xbar,\zbar)}$ by $g^*(\theta)\coloneqq\cg{\F_\V(\xbar,\zbar)}(\{(g(a),g(b)) \mid (a,b)\in\theta\})$. Since $\V$ is congruence distributive and $g$ is surjective, $g^*$ is a lattice homomorphism (cf.~\cite[Lemma~1.11]{Bak74}) and, in particular, $g^*(\theta_1\cap\theta_2)=g^*(\theta_1)\cap g^*(\theta_2)=\theta_1\cap\theta_2$. Hence also $\{g(\hat{\si}_1),\dots,g(\hat{\si}_n)\}$ generates $\theta_1\cap \theta_2$, where $g(\hat{\si}_j)\coloneqq (g(s),g(t))$ for each $\hat{\si}_j=(s,t)$, and we let $\pi(\xbar)$ be the conjunction of the equations $\si_1(\xbar,u,\dots,u),\dots,\si_n(\xbar,u,\dots,u)$.

For any equation $\ep(\xbar,\ybar)$, we may assume without loss of generality that $\ybar\subseteq\zbar$ and use  Lemma~\ref{l:compact-congr-and-consequence-rel} to obtain
\begin{align*}
\V\models (\f_1\OR \f_2)\IMP \ep
&\iff\V\models \f_1\IMP\ep\:\text{ and }\:\V\models \f_2\IMP \ep\\
&\iff\hat{\ep}\in\theta_1\:\text{ and }\:\hat{\ep}\in\theta_2\\
&\iff\hat{\ep}\in\theta_1\cap \theta_2\\
&\iff\V\models\pi\IMP\ep. & \qed\proofend
\end{align*}
\end{proof}

\begin{remark}
The previous lemma can also be deduced from the fact that a variety is congruence distributive and has the compact intersection property if, and only if, it has equationally definable principal meets (cf.~\cite[Theorem~1.5]{BP86}). 
\end{remark}

Recall next that a class of algebras $\K$ has \emph{first-order definable principal congruences} if there exists a formula $\alpha(x_1,x_2,y_1,y_2)$ satisfying for all  $\m A\in\K$ and $a_1,a_2,b_1,b_2\in A$,
\[
(a_1,a_2)\in\cg{\m{A}}(b_1,b_2) \iff \m A\models \alpha(a_1,a_2,b_1,b_2),
\]
where $\cg{\m{A}}(b_1,b_2)\coloneqq\cg{\m{A}}(\{(b_1,b_2)\})$ is the smallest congruence of $\m A$ containing the pair $(b_1,b_2)$. If $\alpha$ can be chosen to be a quantifier-free formula or conjunction of equations, then $\K$ is said to have, respectively, \emph{quantifier-free definable principal congruences} or \emph{equationally definable principal congruences}. For a variety $\V$, it is known that having equationally definable principal congruences corresponds to a property of the associated consequence relation of $\V$ (often referred to as a deduction theorem) and a property of the join-semilattices of compact congruences of members of $\V$.

\begin{proposition}[{\cite[Theorems~5,~8]{KP1980} and~\cite[Theorem~1.7]{BKP84}}]\label{p:EDPC-reform}
The following statements are equivalent for any variety $\V$:
\begin{enumerate}
\item [(1)] $\V$ has equationally definable principal congruences.
\item [(2)] There exists a conjunction of equations $\f(x_1,x_2,y_1,y_2)$ such that for any terms $s_1,s_2,t_1,t_2$ and conjunction of equations $\pi$,
\[
\V \models (\pi \AND (s_1\eq s_2)) \IMP (t_1 \eq t_2) \iff 
\V \models \pi \IMP \f(s_1,s_2,t_1,t_2).
\]
\item [(3)] The join-semilattice of compact congruences of each $\m{A}\in\V$ is dually Brouwerian, i.e., for any compact congruences $\theta_1,\theta_2$ of $\m{A}$, there exists a compact congruence $\theta_1-\theta_2$ of $\m{A}$ such that for every compact congruence $\theta_3$ of $\m{A}$,
\[
\theta_1-\theta_2 \subseteq \theta_3\iff\theta_1 \subseteq \theta_1\jn\theta_2.
\]
\end{enumerate}
\end{proposition}

Equationally definable principal congruences imply two further useful properties.

\begin{proposition}[{\cite[Corollary~2]{Day1971} and~\cite[Corollary~6]{KP1980}}]\label{p:EDPC-implies}
If a variety has equationally definable principal congruences, then it has the congruence extension property and is congruence distributive.
\end{proposition}

%Note that, by Proposition~\ref{p:EDPC-reform}, a variety $\V$ has the compact intersection property and equationally definable principal congruences if, and only if, the compact congruences of each $\m{A}\in\V$ ordered by set-theoretic inclusion form a dually Brouwerian lattice. 

We say next that a class of algebras $\K$ has the {\em equational variable restriction property} if for any finite set $\xbar,y$ and conjunction of equations $\ga(\xbar,y)$, there exists a formula $\pi(\xbar)$ that is either $\bot$ or a conjunction of equations such that $\K\models \pi\IMP \ga$ and for any conjunction of equations $\f(\xbar)$,
\[
\K\models\f\IMP\ga \:\Longrightarrow\: \K\models \f\IMP \pi.
\]
The main result of this section is the following characterization theorem.

\begin{theorem}\label{t:GZ-charact}
Let $\V$ be a variety that has the compact intersection property and equationally definable principal congruences. The theory of $\V$ has a model completion if, and only if, $\V$ is coherent and has the amalgamation property and equational variable restriction property.
\end{theorem}

\begin{remark}
Theorem~\ref{t:GZ-charact} provides a slightly more general algebraic reformulation of~\cite[Theorem~4]{GZ97}, which requires that the join-semilattice of compact congruences of each algebra $\m{A}\in\V$ has a bottom element, or, equivalently, that there is a variable-free conjunction of equations $\pi$ such that $\V\models\pi\IMP\ep$ for every equation $\ep$. This condition is not satisfied by all the varieties of interest in this paper, in particular, the variety of lattice-ordered abelian groups.
\end{remark}

We show first that the right-to-left direction of Theorem~\ref{t:GZ-charact} holds even in the absence of the  compact intersection property.

\begin{proposition}\label{p:suffconditionsedpc}
Let $\V$ be a variety that has equationally definable principal congruences. If $\V$ is coherent and has the amalgamation property and the equational variable restriction property, then the theory of $\V$ has a model completion.
\end{proposition}
\begin{proof}
Using Theorem~\ref{t:charact-model-completion}, it suffices to prove that $\V$ has the conservative model extension property. Consider a finite set $\xbar,y$ and conjunction of literals $\ps(\xbar,y)$, assuming without loss of generality that $\ps$ is satisfiable in $\V$ (cf.~Remark~\ref{r:satisfiable-and-non-empty}). By coherence, there exists a conjunction of equations $\f(\xbar)$ such that $\V\models \ps^+ \IMP \f$ and for any equation $\ep(\xbar)$,
\[
\V\models \ps^+\IMP \ep \:\Longrightarrow\: \V\models \f\IMP \ep.
\]
Let $\si_1,\dots,\si_m$ be the equations of $\ps^-$. Since $\V$ has equationally definable principal congruences, it follows by Proposition~\ref{p:EDPC-reform} that for each $i\in\{1,\dots,m\}$, there exists a conjunction of equations $\pi_i(\xbar,y)$ such that for any conjunction of equations $\ga(\xbar,y)$,
\begin{equation}\label{eq:pi-i}
\V\models (\ga \AND \ps^+)\IMP \si_i \iff \V\models \ga\IMP \pi_i.
\end{equation}
The equational variable restriction property yields a formula $\pi'_i(\xbar)$ for each $i\in\{1,\dots,m\}$ that is either $\bot$ or a conjunction of equations such that $\V\models \pi'_i\IMP \pi_i$ and for any conjunction of equations $\f'(\xbar)$,
\begin{equation}\label{eq:pi-i-primes}
\V\models \f'\IMP \pi_i  \:\Longrightarrow\: \V\models \f'\IMP \pi'_i.
\end{equation}
We claim that the quantifier-free formula $\chi\coloneqq \f \AND\neg \pi'_1\AND\cdots\AND\neg \pi'_m$ satisfies the conditions in the definition of conservative model extension property.

Note first that, since $\V\models \ps\IMP\ps^+$ and $\V\models \ps^+\IMP \f$,  also $\V\models \ps\IMP \f$. To conclude that $\V\models \ps\IMP \chi$, it remains to show that $\V\models \ps\IMP \neg \pi'_i$ for each $i\in\{1,\dots,m\}$. Let $i\in\{1,\dots,m\}$. Since $\V\models \pi'_i\IMP \pi_i $ and $\V\models (\pi_i \AND \ps^+)\IMP \si_i$, also $\V\models (\pi'_i \AND \ps^+)\IMP \si_i$. But then $\V\models (\neg\si_i\AND \ps^+)\IMP\neg\pi'_i$ and hence $\V\models \ps\IMP \neg \pi'_i$ as required.

Now, consider an algebra $\m A\in\V$ generated by $\abar\in A^{\xbar}$ such that $\m A\models \chi(\abar)$ and for any equation $\ep(\xbar)$, 
\begin{equation}\label{eq:cmep-assumption}
\V\models \ps^+ \IMP \ep \:\Longrightarrow\: \m A\models \ep(\abar).
\end{equation}
Let $f\colon \F_{\V}(\xbar)\onto \m A$ be the surjective homomorphism mapping the generators of $\F_{\V}(\xbar)$ to the elements of $\abar$. Considering $\F_{\V}(\xbar)$ as a subalgebra of $\F_{\V}(\xbar,y)$, we define
\[
\theta'\coloneqq\cg{\F_{\V}(\xbar,y)}(\{\hat{\si}\mid \si \text{ is an equation of } \ps^+\})
\enspace\text{and}\enspace
\theta\coloneqq\theta'\jn\cg{\F_{\V}(\xbar,y)}(\ker{f}).
\] 
Let $\m B\coloneqq \F_{\V}(\xbar,y)/\theta$ and let $g$ be the natural homomorphism from $\F_{\V}(\xbar,y)$ onto $\m B$. Since $\ker{f}\subseteq \theta \cap\Fc_{\V}(\xbar)^2$, the inclusion $\F_{\V}(\xbar)\into\F_{\V}(\xbar,y)$ yields a homomorphism $\m A\to \m B$ mapping $f(x)$ to $g(x)$ for each $x\in\xbar$, as illustrated by the following diagram:
\[\begin{tikzcd}
\F_{\V}(\xbar) \arrow[hookrightarrow]{r} \arrow[twoheadrightarrow]{d}[swap]{f} & \F_{\V}(\xbar,y) \arrow[twoheadrightarrow]{d}{g} \\
{\m A} \arrow[dashed]{r} & {\m B}
\end{tikzcd}\]
To prove that this homomorphism is an embedding, it suffices to show that $\theta \cap \Fc_\V(\xbar)^2\subseteq \ker{f}$. Since $\V$ has equationally definable principal congruences, by Proposition~\ref{p:EDPC-implies}, it has the congruence extension property. Hence, by Lemma~\ref{l:cepequiv},
\[
\theta \cap \Fc_\V(\xbar)^2
=(\theta'\jn\cg{\F_{\V}(\xbar,y)}(\ker{f}))\cap \Fc_\V(\xbar)^2
=(\theta' \cap \Fc_\V(\xbar)^2)\jn\ker{f}. 
\]
It therefore suffices to prove that $\theta' \cap \Fc_\V(\xbar)^2\subseteq \ker{f}$. Consider any equation $\ep(\xbar)$ such that $\hat{\ep}\in\theta' \cap \Fc_\V(\xbar)^2$. An application of Lemma~\ref{l:compact-congr-and-consequence-rel} yields $\V\models\ps^+\IMP\ep$ and, by~\eqref{eq:cmep-assumption}, we obtain $\m A\models \ep(\abar)$. Hence  $\hat{\ep}\in \ker{f}$ as required.

To conclude the proof, we show that $\m B\models \ps(\abar,b)$, where $b$ is the image under $g$ of the equivalence class of $y$. Since $\theta' \subseteq \theta=\ker{g}$, we have $\m B\models \ps^+(\abar,b)$. Suppose towards a contradiction that $\m B\models \si_i(\abar,b)$ for some $i\in\{1,\dots,m\}$. Then $\hat{\si}_i\in \theta$, and so $\V\models (\ps^+\AND \ga)\IMP \si_i$ for some conjunction $\ga(\xbar)$ of equations  in $\ker{f}$. By~\eqref{eq:pi-i}, we have $\V\models \ga\IMP \pi_i$ and hence, by~\eqref{eq:pi-i-primes}, also $\V\models \ga\IMP \pi'_i$. Together with $\m A\models \ga(\abar)$, this entails $\m A\models \pi'_i(\abar)$, contradicting the fact that $\m A\models \chi(\abar)$.
\end{proof}

We now complete the proof of Theorem~\ref{t:GZ-charact} by establishing the necessity of the stated conditions, recalling that a variety that has equationally definable principal congruences is congruence distributive (Proposition~\ref{p:EDPC-implies}).

\begin{proposition}
Let $\V$ be a congruence distributive variety that has the compact intersection property. If the theory of $\V$ has a model completion, then $\V$ is coherent and has the amalgamation property and equational variable restriction property.
\end{proposition}
\begin{proof}
If the theory of $\V$ has a model completion, then $\V$ is coherent and has the amalgamation property by Theorem~\ref{t:charact-model-completion}. Hence, it remains to settle the equational variable restriction property. To this end, consider a finite set $\xbar,y$ and conjunction of equations $\ga(\xbar,y)$. We must find a formula $\pi(\xbar)$ that is either $\bot$ or a conjunction of equations such that $\V\models \pi\IMP \ga$ and for any conjunction of equations $\f(\xbar)$,
\[
\V\models\f\IMP\ga \:\Longrightarrow\: \V\models \f\IMP \pi.
\]
Let $T\coloneqq \Th(\V)$ be the theory of $\V$, and let $T^*$  be a model completion of $T$. Since $T^*$ has quantifier elimination, there exists a quantifier-free formula $\xi(\xbar)$ satisfying
\[
T^*\vdash \forall \xbar \, [(\forall y. \ga)\BIMP \xi].
\]
So $T^*\vdash \forall \xbar,y \, (\xi\IMP\ga)$ and, since $T$ and $T^*$ are co-theories, $T\vdash \forall \xbar,y \, (\xi\IMP\ga)$. Assume without loss of generality that $\xi=\xi_1\OR \cdots \OR \xi_m$ where each $\xi_i$ is a conjunction of literals. For each $i\in\{1,\dots,m\}$, we have $\V\models \xi_i\IMP \ga$ and so, by the disjunction property for varieties, either $\V\models \xi_i^+\IMP \ga$ or $\V\models\neg\xi_i$.

Now let $J\coloneqq \{i\in\{1,\dots,m\}\mid \xi_i\text{ is satisfiable in $\V$}\}$.  Then $\V\models\bigOR_{i\in J} \xi^+_i\IMP\ga$, noting that $\bigOR_{i\in J}\xi^+_i=\bot$ for $J=\emptyset$. By Lemma~\ref{l:consequence-of-CIP}, or taking  $\bot$ if $J=\emptyset$, there exists a formula $\pi(\xbar)$ that is either $\bot$ or a conjunction of equations such that for any equation $\ep(\xbar,\zbar)$,
\[
\V\models\bigOR_{i\in J} \xi^+_i \IMP\ep\iff\V\models \pi\IMP \ep.
\]
In particular, $\V\models \pi\IMP \ga$. Now let  $\f(\xbar)$ be any conjunction of equations such that $\V\models \f\IMP \ga$. Then $T\vdash \forall \xbar \, [\f \IMP(\forall y.\ga)]$ and, since $T$ and $T^*$ are co-theories, $T^*\vdash \forall \xbar \, [\f \IMP(\forall y.\ga)]$. Hence $T^*\vdash \forall \xbar \, (\f \IMP \xi)$ and, again using the fact that  $T$ and $T^*$ are co-theories, $T\vdash \forall \xbar \, (\f \IMP \xi)$, yielding $\V\models\f \IMP \xi$. But then $\V\models\f \IMP \bigOR_{i\in J} \xi^+_i$ and, since $\V\models\bigOR_{i\in J} \xi^+_i \IMP\pi$, by the above equivalence, $\V\models \f\IMP\pi$ as required. 
\end{proof}

Finally, in this section, we show that for varieties that have the congruence extension property, a small generalization of the equational variable restriction property implies the amalgamation property. Let us first recall the following well-known relationship between amalgamation and deductive interpolation.

\begin{proposition}[{cf.~\cite[Theorem~22]{MMT14}}]\label{p:amalg:dep}
Let $\V$ be a variety that has the congruence extension property. $\V$ has the amalgamation property if, and only if, it has the {\em deductive interpolation property}; that is, for any finite sets $\xbar,\ybar,\zbar$ and conjunctions of equations $\f(\xbar,\ybar)$, $\ga(\xbar,\zbar)$ satisfying $\V\models\f\IMP\ga$, there exists a conjunction of equations $\pi(\xbar)$ such that $\V\models\f\IMP\pi$ and $\V\models\pi\IMP\ga$.
\end{proposition}

It follows easily that the amalgamation property and equational variable restriction property can be combined into a single ``uniform interpolation'' property.

\begin{corollary}\label{c:amalg+leftvar}
Let $\V$ be a variety that has the congruence extension property. Then the following statements are equivalent:
\begin{enumerate}
\item[\rm (1)] $\V$ has the amalgamation property and the equational variable restriction property.
\item[\rm (2)] For any finite set $\xbar,y$ and conjunction of equations $\ga(\xbar,y)$, there exists a formula $\pi(\xbar)$ that is either $\bot$ or a conjunction of equations such that $\V\models \pi\IMP \ga$ and for any conjunction of equations $\f(\xbar,\zbar)$,
\[
\V\models\f\IMP\ga \:\Longrightarrow\: \V\models \f\IMP \pi.
\]
\end{enumerate}
\end{corollary}
\begin{proof}
$(1)\Rightarrow(2)$ Suppose that $\V$ has the amalgamation property and the equational variable restriction property, and consider a finite set $\xbar,y$ and conjunction of equations $\ga(\xbar,y)$. By the equational variable restriction property,  there exists a formula $\pi(\xbar)$ that is either $\bot$ or a conjunction of equations such that $\V\models \pi\IMP \ga$ and for any conjunction of equations $\f(\xbar)$,
\[
\V\models\f\IMP\ga \:\Longrightarrow\: \V\models \f\IMP \pi.
\]
Now consider a conjunction of equations $\f(\xbar,\zbar)$ satisfying $\V\models\f\IMP\ga$. Since $\V$ has the congruence extension property and the amalgamation property, by Proposition~\ref{p:amalg:dep}, there exists a conjunction of equations $\pi'(\xbar)$ such that $\V\models\f\IMP\pi'$ and $\V\models\pi'\IMP\ga$. By the above implication, $\V\models\pi'\IMP\pi$ and hence also $\V\models\f\IMP\pi$ as required.

$(2)\Rightarrow(1)$ The amalgamation property and the equational variable restriction property both follow directly from condition (2) and Proposition~\ref{p:amalg:dep}.
\end{proof}

Relationships between uniform interpolation and the existence of model completions have been investigated in some depth by Ghilardi and Zawadowski in the setting of intermediate and modal logics~\cite{GZ97,GZ02}. In particular, Pitts' uniform interpolation theorem for intuitionistic propositional logic~\cite{Pit92} is used to deduce that the variety $\HA$ of Heyting algebras satisfies certain categorical conditions and therefore has a model completion. In the terminology of this section: it follows from Pitts'  theorem that $\HA$ is coherent and satisfies condition (2) of Corollary~\ref{c:amalg+leftvar} and hence has the amalgamation property and the equational variable restriction property; the result then follows by Theorem~\ref{t:GZ-charact} (a slight generalization of~\cite[Theorem~4]{GZ97}), since $\HA$ has the compact intersection property and equationally definable principal congruences. More generally, the theories of precisely eight varieties of Heyting algebras (those that have the amalgamation property) have a model completion~\cite{GZ02}; finite axiomatizations of these model completions for the six locally finite cases are provided in~\cite{DJ2018}.

%%%%%%%%%%%%%%%%%%%%%%%%%%%%%%%%%%%%%%%%%%

\section{Parametrically definable principal congruences}\label{s:pdpc}

In this section, we show that if a variety satisfies certain congruence lattice conditions and has a model completion, then it must have equationally definable principal congruences. Our main technical tool is a weaker condition for defining principal congruences that is inspired by work of Glass and Pierce on existentially complete lattice-ordered abelian groups~\cite{GP80}.

We say that a class of algebras $\K$ has \emph{parametrically definable principal congruences} if there exists a quantifier-free formula $\xi(x_1,x_2,y_1,y_2,\zbar)$ such that for each $\m A\in\K$ and all $a_1,a_2,b_1,b_2\in A$,
\[
(a_1,a_2)\in \cg{\m{A}}(b_1,b_2) \iff \text{any } \m B\in\K \text{ extending } \m A \text{ satisfies }  \m B\models \forall \zbar. \xi(a_1,a_2,b_1,b_2,\zbar).
\]
Clearly, if a class of algebras has equationally definable principal congruences, it has parametrically definable principal congruences. However, there are classes of algebras that have parametrically definable principal congruences but do not even have first-order definable principal congruences (see Example~\ref{ex:ablggroupsnomodelcompletion}). The following proposition shows that this cannot be the case for a universal class of algebras whose theory has a model completion.

\begin{proposition}\label{p:splitting-lemma-fo}
Let $\K$ be a universal class of algebras with parametrically definable principal congruences. If the theory of $\K$ has a model completion, then $\K$ has quantifier-free definable principal congruences.
\end{proposition}
\begin{proof}
Suppose that $\K$ has parametrically definable principal congruences, witnessed by a quantifier-free formula $\xi(x_1,x_2,y_1,y_2,\zbar)$, and that the theory of $\K$ has a model completion. We can assume that $\neg\xi=\ps_1\OR\cdots\OR \ps_n$, where each $\ps_i$ is a conjunction of literals. By Proposition~\ref{p:strong-CMEP-iff}, for each $i\in\{1,\dots,n\}$, there exists a quantifier-free formula $\chi_i(x_1,x_2,y_1,y_2)$ such that
  \begin{enumerate}
  \item[\rm (i)] $\K\models \ps_i\IMP\chi_i$
  \item[\rm (ii)] for any $\m A\in\K$, and $a_1,a_2,b_1,b_2\in A$ satisfying $\m A\models\chi_i(a_1,a_2,b_1,b_2)$, there exists an extension $\m B\in\K$ of $\m A$ such that $\m B\models \exists\zbar.\ps_i(a_1,a_2,b_1,b_2,\zbar)$. 
 \end{enumerate}
We claim that for each $\m A\in\K$ and all $a_1,a_2,b_1,b_2\in A$,
\begin{equation*}
(a_1,a_2)\in\cg{\m A}(b_1,b_2) \iff \m A\models \bigAND_{1\leq i\leq n}\neg\chi_i(a_1,a_2,b_1,b_2),
\end{equation*}
and hence that $\K$ has quantifier-free definable principal congruences. Suppose first that $(a_1,a_2)\in\cg{\m A}(b_1,b_2)$. By assumption, any $\m B\in\K$ extending $\m A$ satisfies $\m{B}\models\forall \zbar. \xi(a_1,a_2,b_1,b_2,\zbar)$ and hence $\m B\not\models \exists \zbar. \ps_i(a_1,a_2,b_1,b_2,\zbar)$ for each $i\in\{1,\dots,n\}$. But then, by~(ii), also $\m A\models \neg\chi_i(a_1,a_2,b_1,b_2)$ for each $i\in\{1,\dots,n\}$. Now suppose that $(a_1,a_2)\notin \cg{\m A}(b_1,b_2)$. By assumption, there exists an extension $\m B\in\K$ of $\m A$ such that $\m B\not\models\forall \zbar. \xi(a_1,a_2,b_1,b_2,\zbar)$, and hence $\m B\models \exists \zbar. \ps_i(a_1,a_2,b_1,b_2,\zbar)$ for some $i\in\{1,\dots,n\}$. But then, since $\K\models \ps_i\IMP\chi_i$, by (i), also $\m B\models \chi_i(a_1,a_2,b_1,b_2)$. Finally, as $\chi_i$ is quantifier-free, it follows that $\m A\models \chi_i(a_1,a_2,b_1,b_2)$.
\end{proof}

\begin{example}
It is easy to see that the variety of abelian groups does not have first-order definable principal congruences. If this were the case, there would be a first-order formula $\alpha(x,y)$ such that for any abelian group $\m G$, an element $a\in G$ belongs to the  subgroup of $\m G$ generated by an element $b\in G$ if, and only if, $\m G\models \alpha(a,b)$. But then the sentence $\exists y \forall x. \alpha(x,y)$ would define the class of cyclic groups, contradicting the fact that this class is not elementary. On the other hand, the theory of abelian groups does have a model completion (cf.~\cite{ES70}). Proposition~\ref{p:splitting-lemma-fo} therefore tells us that the variety of abelian groups does not have parametrically definable principal congruences.
\end{example}

We strengthen Proposition~\ref{p:splitting-lemma-fo} by exploiting the following fact, due to Fried, Gr{\"a}tzer, and Quackenbush~\cite{FGQ80}. 

\begin{proposition}[{\cite[Theorems~4.5 and~6.9]{FGQ80}}]\label{p:fo-edpc}
Let $\V$ be a congruence distributive variety that has the congruence extension property. $\V$ has first-order definable principal congruences if, and only if, it has equationally definable principal congruences.
\end{proposition}

Combining this fact with Proposition~\ref{p:splitting-lemma-fo} yields the following result. 

\begin{proposition}\label{p:splitting-lemma-edpc}
Let $\V$ be a congruence distributive variety that has the congruence extension property and parametrically definable principal congruences. If the theory of $\V$ has a model completion, then $\V$ has equationally definable principal congruences.
\end{proposition}

We now provide a sufficient syntactic condition for a variety with the congruence extension property to have parametrically definable principal congruences that is easier to establish in certain cases. We say that a variety $\V$ has a \emph{guarded deduction theorem} if there exist conjunctions of equations $\ga(x_1,x_2,y_1,y_2,\zbar)$ and $\f(x_1,x_2,y_1,y_2,\zbar)$ such that for any finite set $\wbar$ with $\wbar\cap \zbar=\emptyset$, terms $s_1(\wbar),s_2(\wbar),t_1(\wbar),t_2(\wbar)$, and conjunction of equations $\pi(\wbar)$, 
\begin{enumerate}
\item[\rm (i)] $\V \models (\pi \AND (t_1\eq t_2)) \IMP (s_1 \eq s_2) \iff \V \models \pi \IMP \forall\zbar.(\ga\IMP \f)(s_1,s_2,t_1,t_2,\zbar)$
\item[\rm (ii)] $\V \models (\pi \AND \ga(s_1,s_2,t_1,t_2,\zbar)) \IMP \si \iff \V\models \pi \IMP \si$ \,for any equation $\si(\wbar)$.
\end{enumerate}
%By Proposition~\ref{p:EDPC-reform}, every variety that has equationally definable principal congruences has a guarded deduction theorem. 
This property has the following algebraic characterization:

\begin{proposition}\label{p:guarded-algebraically}
The following statements are equivalent for any variety $\V$:
\begin{enumerate}
\item[(1)] $\V$ has a guarded deduction theorem.
\item[(2)] There exist conjunctions of equations $\ga(x_1,x_2,y_1,y_2,\zbar)$ and $\f(x_1,x_2,y_1,y_2,\zbar)$ such that for any finitely generated algebra $\m A\in \V$ and $a_1,a_2,b_1,b_2\in A$, there exists an embedding of $\m A$ into some member of $\V$ satisfying $\exists\zbar.\ga(a_1,a_2,b_1,b_2,\zbar)$, and 
\begin{align*}
(a_1,a_2)\in \cg{\m A}(b_1,b_2) \iff\: & \m B\models \forall \zbar. (\ga\IMP\f)(h(a_1),h(a_2),h(b_1),h(b_2),\zbar)\\
& \text{for any $\m B\in \V$ and homomorphism $h\colon \m A\to\m B$.}
\end{align*}
\end{enumerate}
\end{proposition}
\begin{proof}
 $(1)\Rightarrow(2)$ Let $\ga(x_1,x_2,y_1,y_2,\zbar)$ and $\f(x_1,x_2,y_1,y_2,\zbar)$ be conjunctions of equations satisfying the conditions for a guarded deduction theorem, and fix a finitely generated algebra $\m A\in \V$ and $a_1,a_2,b_1,b_2\in A$. We claim that $\ga$ and $\f$ satisfy the equivalence given in (2). 

Note first that since $\m{A}$ is finitely generated, there exists a surjective homomorphism $f\colon\Tm(\wbar)\onto \m A$ for some finite set $\wbar$ with $\wbar\cap\zbar=\emptyset$. Let $s_1,s_2,t_1,t_2$ be terms in the preimages under $f$ of $a_1,a_2,b_1,b_2$, respectively. Now consider $(a_1,a_2)\in \cg{\m A}(b_1,b_2)$. Since $\m A$ is isomorphic to $\Tm(\wbar)/\ker{f}$, by the correspondence theorem for universal algebra, $(s_1,s_2)\in \ker{f} \jn \cg{\Tm(\wbar)}(t_1,t_2)$. Moreover, since any congruence is the directed union of the compact congruences below it, there exists a finite subset $S\subseteq \ker{f}$ such that $(s_1,s_2)\in \cg{\Tm(\wbar)}(S) \jn \cg{\Tm(\wbar)}(t_1,t_2)$. Let $\pi(\wbar)$ be the conjunction of the equations in $S$. It follows easily from Lemma~\ref{l:compact-congr-and-consequence-rel} that $\V \models (\pi \AND (t_1\eq t_2)) \IMP (s_1 \eq s_2)$. Hence $\V \models \pi \IMP \forall\zbar.(\ga\IMP \f)(s_1,s_2,t_1,t_2,\zbar)$, by condition (i) in the definition of a guarded deduction theorem. But $\m A \models \pi(f(\wbar))$, so for any $\m B\in \V$ and homomorphism $h\colon \m A\to \m B$,
\[
\m B\models \forall \zbar. (\ga\IMP\f)(h(a_1),h(a_2),h(b_1),h(b_2),\zbar).
\] 
Now suppose that $\m B\models \forall \zbar. (\ga\IMP\f)(h(a_1),h(a_2),h(b_1),h(b_2),\zbar)$ for any $\m B\in \V$ and homomorphism $h\colon \m A\to\m B$. Let $\Diag(\m A)$ be the positive diagram of $\m A$, i.e., the set of all atomic sentences in the language extended with names for the elements of $\m A$ that are satisfied in $\m A$. Then 
\[
\Th(\V)\cup\Diag(\m A)\vdash \forall\zbar.(\ga\IMP \f)(a_1,a_2,b_1,b_2,\zbar),
\] 
and hence, by the compactness theorem for first-order logic, there exists a finite subset $\Si\subseteq \Diag(\m A)$ such that $\Th(\V)\cup\Si\vdash \forall\zbar.(\ga\IMP \f)(a_1,a_2,b_1,b_2,\zbar)$. For each member of $\Si$, consider an equation in its preimage under the surjection $\Tm(\wbar)^2\onto \m A^2$. This yields a finite subset of $\ker{f}$, and letting $\pi(\wbar)$ denote the conjunction of the equations in this set, we obtain $\V \models \pi \IMP \forall\zbar.(\ga\IMP \f)(s_1,s_2,t_1,t_2,\zbar)$. By condition (i) in the definition of a guarded deduction theorem, we obtain $\V \models (\pi \AND (t_1\eq t_2)) \IMP (s_1 \eq s_2)$. Now let $q\colon \m A\onto \m A/\cg{\m A}(b_1,b_2)$ be the natural quotient map. Since $\m A/\cg{\m A}(b_1,b_2), qf\models \pi \AND (t_1\eq t_2)$, we have $\m A/\cg{\m A}(b_1,b_2), q f\models s_1 \eq s_2$, i.e., $(a_1,a_2)\in \cg{\m A}(b_1,b_2)$. 

To conclude, it remains to show that $\m A$ embeds into some member of $\V$ satisfying $\exists \zbar.\ga(a_1,a_2,b_1,b_2,\zbar)$. Let $\m B$ be the quotient of $\F_{\V}(\wbar,\zbar)$ with respect to the congruence $\cg{\F_{\V}(\wbar,\zbar)}(\{\hat{\ep}\mid \ep\in \ker{f}\}\cup\{\hat{\si}\mid \si \text{ is an equation of }\ga(s_1,s_2,t_1,t_2,\zbar)\})$. Then $\m B$ satisfies $\exists \zbar.\ga(a_1,a_2,b_1,b_2,\zbar)$ and it is not difficult to see, using condition (ii) in the definition of a guarded deduction theorem, that $\m B$ extends $\m A$.

$(2)\Rightarrow(1)$ Let $\ga(x_1,x_2,y_1,y_2,\zbar)$ and $\f(x_1,x_2,y_1,y_2,\zbar)$ be conjunctions of equations satisfying the conditions in (2). Let $\wbar$ be any finite set satisfying $\wbar\cap\zbar=\emptyset$ and consider terms $s_1(\wbar),s_2(\wbar),t_1(\wbar),t_2(\wbar)$ and a conjunction of equations $\pi(\wbar)$.

Let $\m A$ be the quotient of $\F_{\V}(\wbar)$ with respect to $\cg{\F_{\V}(\wbar)}(\{\hat{\ep}\mid \ep \text{ is an equation of } \pi\})$, with quotient map $f\colon \F_{\V}(\wbar) \onto \m A$. Let $q\colon \Tm(\wbar)\onto \F_{\V}(\wbar)$ denote the natural quotient map and define $a_1\coloneqq f(q(s_1))$, $a_2\coloneqq f(q(s_2))$, $b_1\coloneqq f(q(t_1))$, and $b_2\coloneqq f(q(t_2))$. By Lemma~\ref{l:compact-congr-and-consequence-rel}, $(a_1,a_2)\in \cg{\m A}(b_1,b_2)$ if, and only if, $\V \models (\pi \AND (t_1\eq t_2)) \IMP (s_1 \eq s_2)$.  Hence, in order to settle condition (i) in the definition of a guarded deduction theorem, it remains to show that $\V \models \pi \IMP \forall\zbar.(\ga\IMP \f)(s_1,s_2,t_1,t_2,\zbar)$ if, and only if, for any $\m B\in \V$ and homomorphism $h\colon \m A\to \m B$,
\[
\m B\models \forall \zbar. (\ga\IMP\f)(h(a_1),h(a_2),h(b_1),h(b_2),\zbar).
\]
This follows by reasoning as in the proof of $(1)\Rightarrow (2)$.

With respect to condition (ii) in the definition of a guarded deduction theorem, let $\m B$ denote the quotient of $\F_{\V}(\wbar,\zbar)$ with respect to 
\[
\cg{\F_{\V}(\wbar,\zbar)}(\{\hat{\ep}\mid \ep \text{ is an equation of }\pi(\wbar)\}\cup\{\hat{\si}\mid \si \text{ is an equation of }\ga(s_1,s_2,t_1,t_2,\zbar)\}).
\]
Note that there is a canonical homomorphism $k\colon \m A\to \m B$. By assumption, there is an embedding $j\colon \m A \into \m B'$ with $\m B'\in \V$ and $\m B'\models \exists \zbar.\ga(a_1,a_2,b_1,b_2,\zbar)$. It follows easily that there is a homomorphism $h\colon \m B\to \m B'$ such that $j=h k$. Since $j$ is injective, so is $k$. Hence, by Lemma~\ref{l:compact-congr-and-consequence-rel}, we have that $\V \models (\pi \AND \ga(s_1,s_2,t_1,t_2,\zbar)) \IMP \si'$ implies $\V\models \pi \IMP \si'$ for any equation $\si'(\wbar)$.
\end{proof}

\begin{proposition}\label{p:guarded-entails-parametrically}
Let $\V$ be a variety that has the congruence extension property and a guarded deduction theorem. Then $\V$ has parametrically definable principal congruences.
\end{proposition}
\begin{proof}
It suffices to show that $\V$ has parametrically definable principal congruences whenever the property in condition (2) of Proposition~\ref{p:guarded-algebraically} holds for some $\ga$ and $\f$. Let $\xi(x_1,x_2,y_1,y_2,\zbar)\coloneqq \ga\IMP \f$. We claim that for all $\m A\in\V$ and $a_1,a_2,b_1,b_2\in A$,
\[
(a_1,a_2)\in \cg{\m{A}}(b_1,b_2) \iff \text{each } \m B\in\V \text{ extending } \m A \text{ satisfies }  \m B\models \forall \zbar. \xi(a_1,a_2,b_1,b_2,\zbar).
\]
Suppose first that $(a_1,a_2)\in \cg{\m{A}}(b_1,b_2)$. By the congruence extension property, $(a_1,a_2)\in \cg{\m A'}(b_1,b_2)$, where $\m A'$ is the subalgebra of $\m A$ generated by $a_1,a_2,b_1,b_2$. Since $\m A'$ is a finitely generated member of $\V$, it follows from condition (2) of Proposition~\ref{p:guarded-algebraically} that $\m B\models \forall \zbar. \xi(a_1,a_2,b_1,b_2,\zbar)$ for every $\m B\in \V$ extending $\m A'$, and hence, in particular, for every $\m B\in \V$ extending $\m A$.

Now suppose that $(a_1,a_2)\notin \cg{\m{A}}(b_1,b_2)$. We assume first that $\m A$ is finitely generated, and then deduce the general case. Let $\wbar$ be a finite set such that $\wbar\cap\zbar=\emptyset$ and there exists a surjective homomorphism $f\colon\Tm(\wbar)\onto \m A$. We can assume without loss of generality that $x_1,x_2,y_1,y_2\in \wbar$, and also $f(x_1)=b_1, f(x_2)=b_2, f(y_1)=a_1, f(y_2)=a_2$. Let $\m B\coloneqq \F_{\V}(\wbar,\zbar)/\theta$, where 
\[
\theta\coloneqq\cg{\F_{\V}(\wbar,\zbar)}(\{\hat{\ep}\mid \ep\in \ker{f}\}\cup\{\hat{\si}\mid \si \text{ is an equation of } \ga\}).
\]
We claim that $\m B$ extends $\m A$. Let $j\colon \m A\to \m B$ be the canonical homomorphism  mapping $a_1,a_2,b_1,b_2$ to the equivalence classes of $x_1,x_2,y_1,y_2$, respectively. By assumption, there exist $\m B'\in \V$ and an embedding $i\colon \m A\to \m B'$ such that $\m B'\models \exists\zbar.\ga(a_1,a_2,b_1,b_2,\zbar)$. It is not difficult to see that $i$ factors through $j$, and since $i$ is injective, so is $j$.

Now, by assumption, there exist $\m C\in \V$ and a homomorphism $h\colon \m A\to \m C$ satisfying $\m C\models \exists \zbar. (\ga \AND \neg \f_i)(h(a_1),h(a_2),h(b_1),h(b_2),\zbar)$ for some equation $\f_i$ of $\f$. It follows that there exists a homomorphism $k\colon \Tm(\wbar,\zbar)\to \m C$ which factors through the composition $g\colon \Tm(\wbar,\zbar)\onto \F_{\V}(\wbar,\zbar)\onto \m B$ and satisfies $\f_i\notin \ker{k}$, as depicted in the diagram:
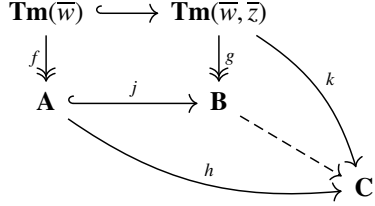
\[
\begin{tikzcd}
\Tm(\wbar) \arrow[hookrightarrow]{r} \arrow[twoheadrightarrow]{d}[swap]{f} & \Tm(\wbar,\zbar) \arrow[twoheadrightarrow]{d}{g} \arrow[bend left=20]{ddr}{k} & {} \\
{\m A} \arrow[bend right=20]{drr}{h} \arrow[hookrightarrow]{r}{j} & {\m B} \arrow[dashed]{dr} & {} \\
{} & & {\m C}
\end{tikzcd}
\]
So $\f_i\notin \ker{g}$, showing that $\m B$ is an extension of $\m A$ satisfying $\m B\not\models \forall \zbar. \xi(a_1,a_2,b_1,b_2,\zbar)$.

To conclude, consider any $\m A\in\V$ and suppose that $\m B\models \forall \zbar. \xi(a_1,a_2,b_1,b_2,\zbar)$ for all $\m B\in \V$ extending $\m A$. Then $\Th(\V)\cup\mathfrak{D}(\m A)\vdash \forall \zbar. \xi(a_1,a_2,b_1,b_2,\zbar)$, where $\mathfrak{D}(\m A)$ is the diagram of $\m A$, i.e., the collection of all atomic sentences and negated atomic sentences in the language extended with names for the elements of $\m A$ that are satisfied in $\m A$. By the compactness theorem for first-order logic, there is a finite subset $\Si\subseteq \mathfrak{D}(\m A)$ such that $\Th(\V)\cup\Si\vdash \forall \zbar. \xi(a_1,a_2,b_1,b_2,\zbar)$. Let $\m A'$ be the subalgebra of $\m A$ generated by $a_1,a_2,b_1,b_2$ and the elements named by $\Si$. As the diagram of $\m A'$ contains $\Si$, every $\m B'\in \V$ extending $\m A'$ satisfies $\m B'\models \forall \zbar. \xi(a_1,a_2,b_1,b_2,\zbar)$. Note that $\m A'$ is a finitely generated member of $\V$ and so, by the argument above, $(a_1,a_2)\in \cg{\m A'}(b_1,b_2)\subseteq \cg{\m A}(b_1,b_2)$.
\end{proof}

Combining Propositions~\ref{p:splitting-lemma-edpc} and~\ref{p:guarded-entails-parametrically} yields the main result of this section.

\begin{theorem}\label{t:mc-implies-EDPC-guarded}
Let $\V$ be a congruence distributive variety that has the congruence extension property and a guarded deduction theorem. If the theory of $\V$ has a model completion, then $\V$ has equationally definable principal congruences.
\end{theorem}

\begin{example}\label{ex:ablggroupsnomodelcompletion}
As first proved by Glass and Pierce in~\cite{GP80}, the theory of the variety $\LA$ of lattice-ordered abelian groups does not have a model completion. It is well-known that $\LA$ is congruence distributive and has the congruence extension property, but does not have equationally (or even first-order) definable principal congruences. Hence it suffices, by Theorem~\ref{t:mc-implies-EDPC-guarded}, to observe that the formulas
\[
\ga\coloneqq (x_1 - x_2) \mt(x_2 - x_1) \mt 0 \leq z \AND ((y_1-y_2)\mt (y_2-y_1) \mt 0)\jn z \eq 0\:\text{ and }\:\f\coloneqq z\eq 0
\]
satisfy conditions (i) and (ii) in the definition of a guarded deduction theorem for $\LA$.
\end{example}

%%%%%%%%%%%%%%%%%%%%%%%%%%%%%%%%%%%%%%%%%%

\section{Varieties of pointed residuated lattices}\label{s:prls}

In this section, we apply the results of the last two sections to a family of varieties of algebras that provide algebraic semantics for substructural logics, including (up to term-equivalence) lattice-ordered groups, MV-algebras, and Heyting algebras (see,~e.g.,~\cite{BT03,GJKO07,MPT10}).

A {\em pointed residuated lattice} is an algebra $\m{A} = \lan A,\mt,\jn,\cdot,\ld,\rd,\e, \zr\ran$ such that $\lan A,\mt,\jn \ran$ is a lattice, $\lan A,\cdot,\e \ran$ is a monoid, and $\ld,\rd$ are left and right residuals, respectively, of $\cdot$ in the underlying lattice order, i.e., for all $a,b,c\in A$,
\[
b \le a \ld c \iff  a \cdot b \le c  \iff a \le c \rd b.
\]
It will be useful to define a further binary operation $x\eqv y\coloneqq (x\ld y)\mt (y\ld x)\mt \e$, noting that $(a\eqv b) = \e$ for $a,b\in A$ if, and only if, $a=b$. We also define for any $a\in A$ inductively $a^0\coloneqq \e$ and $a^{n+1}\coloneqq a\cdot a^n$  ($n\in\mathbb{N}$).

Pointed residuated lattices form a congruence distributive variety~\cite{BT03}, and include the subvariety $\CPRL$ of commutative pointed residuated lattices satisfying $x \cdot y \eq y \cdot x$. In particular, a Heyting algebra is term-equivalent to a commutative pointed residuated lattice satisfying $x\cdot y\eq x\mt y$ and $x\mt\zr\eq\zr$, and a Boolean algebra is  term-equivalent to a Heyting algebra satisfying $(x \ld \zr)\ld\zr \eq x$.

Let $\V$ be a variety of pointed residuated lattices. We let $\Vtot$ denote the class of linearly ordered members of $\V$, and call $\V$ {\em semilinear} if $\V = \cop{ISP}(\Vtot)$. Semilinearity can also be expressed equationally; in particular, a variety of commutative pointed residuated lattices is semilinear if, and only if, it satisfies $\e \eq ((x \ld y) \mt \e) \jn ((y \ld x) \mt \e)$~\cite{BT03}.

\begin{example}
The variety $\LA$ of lattice-ordered abelian groups is term-equivalent to the variety of commutative pointed residuated lattices satisfying $\e\eq\zr$ and $x\cdot(x\ld\e)\eq~\e$. More precisely, if $\lan L,\mt,\jn,+,-,0 \ran\in\LA$ and we define $x\cdot y\coloneqq x+y$, $x\ld y\coloneqq y-x$, and $x\rd y\coloneqq x-y$, then $\lan L,\mt,\jn,\cdot,\ld,\rd,0,0\ran\in\CPRL$ satisfies $\e\eq\zr$ and $x\cdot(x\ld\e)\eq\e$. Conversely, if $\m{L}\in\CPRL$ satisfies $\e\eq\zr$ and $x\cdot(x\ld\e)\eq\e$ and we define $x+y\coloneqq x\cdot y$ and $-x\coloneqq x\ld\e$, then $\lan L,\mt,\jn,+,-,\zr\ran\in\LA$. Recall also that $\LA = \cop{ISP}(\OG)$, so the corresponding variety of pointed residuated lattices is semilinear.
\end{example}

\begin{example}
The variety $\MV$ of MV-algebras is term-equivalent to the variety of commutative pointed residuated lattices satisfying $x\jn y\eq(x\ld y)\ld y$ and $x\mt\zr\eq\zr$. More precisely, if $\lan M,\oplus,\lnot,0\ran\in\MV$ and we define $x \cdot y \coloneqq \lnot(\lnot x\oplus\lnot y)$, $x \ld y \coloneqq  \lnot x \oplus y$, ${x \rd y} \coloneqq x\oplus\lnot y$, and $\e\coloneqq \lnot 0$, then $\lan M,\mt,\jn,\cdot,\ld,\rd,\e,\zr\ran\in\CPRL$ satisfies $x\jn y\eq(x\ld y)\ld y$ and $x\mt\zr\eq\zr$. Conversely, if $\m{M}\in\CPRL$ satisfies $x\jn y\eq(x\ld y)\ld y$ and $x\mt\zr\eq\zr$ and we define $x\oplus y\coloneqq (x\ld\zr)\ld y$ and $\lnot x\coloneqq x\ld\zr$, then $\lan M,\oplus,\lnot,\zr\ran\in\MV$. Again, since $\MV=\cop{ISP}(\MVC)$, the corresponding variety of pointed residuated lattices is semilinear.
\end{example}

Let us call a variety $\V$ of pointed residuated lattices \emph{Hamiltonian}\footnote{An algebra $\m{A}$ is usually called {\em Hamiltonian} if every non-empty subuniverse of $\m{A}$ is an equivalence class of some congruence of $\m{A}$. In~\cite{BKLT16}, it is shown that a variety of pointed residuated lattices satisfying $x\ld\e\eq\e\rd x$ consists of Hamiltonian algebras in this sense if, and only if, it has the property given in our definition.} if for some $k\in\N^{>0}$,
\[
\V\models(x\mt \e)^k \cdot y\eq y\cdot (x\mt \e)^k. 
\]
The following proposition collects some useful facts about such varieties.

\begin{proposition}[{\cite[Lemma~3.14,~Proposition~3.15,~and~Lemma~3.17]{Gal03}}]
\label{p:hamiltonian-properties}
Let $\V$ be a Hamiltonian variety of pointed residuated lattices. 
\begin{enumerate}
\item[\rm (a)] For any $\m{A}\in\V$ and $a_1,a_2,b_1,b_2\in A$,
\[
(a_1,a_2)\in \cg{\m A}(b_1,b_2) \iff  (b_1\eqv b_2)^n\leq a_1\eqv a_2 \text{ for some }n\in\N.
\]
\item[\rm (b)] $\V$ has equationally definable principal congruences if, and only if, for some $n\in\N$,
\[
\V\models(x\mt\e)^n\eq (x\mt\e)^{n+1}.
\]
\item[\rm (c)] $\V$ has the congruence extension property.
\end{enumerate}
\end{proposition}

We show now that every Hamiltonian variety of pointed residuated lattices has both the compact intersection property and a guarded deduction theorem, which, combined with Theorems~\ref{t:GZ-charact} and~\ref{t:mc-implies-EDPC-guarded}, will allow us to determine which of these varieties have a theory that has a model completion. 

\begin{lemma}
Let $\V$ be a Hamiltonian variety of pointed residuated lattices. Then $\V$ has the compact intersection property.
\end{lemma}
\begin{proof}
Every compact congruence of an algebra $\m{A}\in\V$ is a finite join of principal congruences of $\m{A}$ and hence, by congruence distributivity, the intersection of any two compact congruences of $\m{A}$ is a finite join of intersections of principal congruences of $\m{A}$. It therefore suffices to show that for all $b_1,b_2,c_1,c_2\in A$,
\[
\cg{\m A}(b_1,b_2)\cap\cg{\m A}(c_1,c_2)=\cg{\m A}(\e,(b_1\eqv b_2)\jn (c_1\eqv c_2)).
\]
Suppose first that $(a_1,a_2)\in\cg{\m A}(b_1,b_2)\cap\cg{\m A}(c_1,c_2)$. Then Proposition~\ref{p:hamiltonian-properties}.(a) yields $(b_1\eqv b_2)^{n_1}\leq a_1\eqv a_2$ and $(c_1\eqv c_2)^{n_2}\leq a_1\eqv a_2$ for some $n_1,n_2\in\N$. Let $n\coloneqq \max{(n_1,n_2)}$.  Then $(b_1\eqv b_2)^{n} \jn (c_1\eqv c_2)^{n}\leq a_1\eqv a_2$ and, using some elementary properties of pointed residuated lattices, also $((b_1\eqv b_2)\jn (c_1\eqv c_2))^{2n}\leq {a_1\eqv a_2}$. By Proposition~\ref{p:hamiltonian-properties}.(a) again, $(a_1,a_2)\in \cg{\m A}(\e,(b_1\eqv b_2)\jn (c_1\eqv c_2))$. Now suppose that $(a_1,a_2)\in \cg{\m A}(\e,(b_1\eqv b_2)\jn (c_1\eqv c_2))$. It follows easily, again using some elementary properties of pointed residuated lattices, that $(a_1,a_2)\in\cg{\m A}(b_1,b_2)\cap\cg{\m A}(c_1,c_2)$. 
\end{proof}

\begin{lemma}\label{l:guarded-HRL}
Let $\V$ be a Hamiltonian variety of pointed residuated lattices. Then $\V$ has a guarded deduction theorem.
\end{lemma}
\begin{proof}
We show that the formulas
\[
\ga\coloneqq (x_1\eqv x_2) \leq z \AND (y_1\eqv y_2)\jn z \eq \e \: \text{ and } \: \f\coloneqq z\eq \e
\]
satisfy conditions (i) and (ii) in the definition of a guarded deduction theorem for a finite set $\wbar$ with $z\notin \wbar$,  terms $s_1(\wbar),s_2(\wbar),t_1(\wbar),t_2(\wbar)$, and conjunction of equations $\pi(\wbar)$.

For (i), suppose first that $\V \models (\pi \AND (t_1\eq t_2)) \IMP (s_1 \eq s_2)$. To show that $\V \models \pi \IMP \forall z.(\ga\IMP \f)(s_1,s_2,t_1,t_2,z)$, consider any $\m A\in \V$ and assignment $f\colon \wbar \to A$ such that $\m A, f\models \pi$, and let $g\colon \wbar, z\to A$ be a map extending $f$ such that $\m A, g\models \ga(s_1,s_2,t_1,t_2,z)$. Then $\m A, g\models  (s_1\eqv s_2) \leq z$ and $\m A, g\models (t_1\eqv t_2)\jn z \eq \e$. It follows that also $\m A, g\models z\leq \e$ and hence to show that $\m A, g\models \f(s_1,s_2,t_1,t_2,z)$, it suffices to prove the following:
\begin{claim*}
$\m A, g\models \e \leq z$.
\end{claim*}
\begin{proof}[Proof of Claim]
Let $\tilde{g}\colon \Tm(\wbar,z)\to \m A$ be the unique homomorphism extending $g$, and define $a_1\coloneqq \tilde{g}(s_1)$, $a_2\coloneqq \tilde{g}(s_2)$, $b_1\coloneqq \tilde{g}(t_1)$, $b_2\coloneqq \tilde{g}(t_2)$, and $c\coloneqq \tilde{g}(z)$. We prove first that $\e\le(b_1\eqv b_2)^k \jn c$ for all $k\in\N^{>0}$, proceeding by induction on $k$. The base case follows from the fact that $\m A, g\models (t_1\eqv t_2)\jn z \eq \e$.  For the inductive step, we obtain
\begin{align*}
\e & \le (b_1\eqv b_2)^k \jn c && \text{by the induction hypothesis}\\
& \le (((b_1\eqv b_2) \jn c)\cdot (b_1\eqv b_2)^k) \jn c && \text{since  $\e\le(b_1\eqv b_2) \jn c$}\\
& = ((b_1\eqv b_2)^{k+1} \jn c\cdot (b_1\eqv b_2)^k) \jn c && \text{by the distributivity of $\cdot$ over $\jn$}\\
& \le (b_1\eqv b_2)^{k+1} \jn c  && \text{since  $b_1\eqv b_2\leq \e$.}
\end{align*}
Now let $q$ denote the natural quotient map from $\m A$ onto $\m A/\cg{\m A}(b_1,b_2)$. By assumption, $\V \models (\pi \AND (t_1\eq t_2)) \IMP (s_1 \eq s_2)$, so $\m A/\cg{\m A}(b_1,b_2), q f\models s_1\eq s_2$, i.e., $(a_1, a_2)\in \cg{\m A}(b_1,b_2)$. By Proposition~\ref{p:hamiltonian-properties}.(a), we obtain $(b_1\eqv b_2)^n\leq a_1\eqv a_2$ for some $n\in\N$. But $\m A, g\models  (s_1\eqv s_2) \leq z$, so also $(a_1\eqv a_2)\leq c$ and $(b_1\eqv b_2)^n\leq c$. Hence we get $\e\le (b_1\eqv b_2)^n \jn c = c$; that is, $\m A, g\models \e \leq z$.
 \end{proof}

Now suppose that $\V \models \pi \IMP \forall\zbar.(\ga\IMP \f)(s_1,s_2,t_1,t_2,\zbar)$. To show that  $\V \models (\pi \AND (t_1\eq t_2)) \IMP (s_1 \eq s_2)$, consider any $\m A\in \V$ and assignment $f\colon \wbar\to A$ such that $\m A, f\models \pi \AND (t_1\eq t_2)$. Since $z\notin \wbar$, we can extend $f$ to an assignment $g\colon \wbar,z\to A$ by setting $g(z)\coloneqq \tilde{f}(s_1\eqv s_2)$, where $\tilde{f}\colon \Tm(\wbar)\to \m A$ is the unique homomorphism extending $f$. Clearly, $\m A, g\models \pi \AND \ga(s_1,s_2,t_1,t_2,z)$ and hence, by assumption, $\m A, g\models z\eq \e$. It follows that $\m A, f\models (s_1\eqv s_2)\eq \e$, and so $\m A, f\models s_1\eq s_2$. 

For the non-trivial direction of (ii), suppose that $\V \models (\pi \AND \ga(s_1,s_2,t_1,t_2,\zbar)) \IMP \si$ for some equation $\si(\wbar)$. To show that $\V\models \pi\IMP\si$, consider any algebra $\m A\in \V$ and assignment $f\colon \wbar \to A$ such that $\m A, f\models \pi$. Extend $f$ to an assignment $g\colon \wbar,z\to A$ by setting $g(z)\coloneqq \e$. Then $\m A, g\models \pi \AND \ga(s_1,s_2,t_1,t_2,z)$ and hence, by assumption, $\m A, f\models \si$. 
\end{proof}

Theorems~\ref{t:GZ-charact} and~\ref{t:mc-implies-EDPC-guarded} then yield the following description of Hamiltonian varieties of pointed residuated lattices whose theories have a model completion.

\begin{theorem}\label{t:HamRL-charact}
Let $\V$ be a Hamiltonian variety of pointed residuated lattices. The theory of $\V$ has a model completion if, and only if, $\V$ is coherent and has equationally definable principal congruences, the amalgamation property, and the equational variable restriction property.
\end{theorem}

\begin{example}
Since the variety of lattice-ordered abelian groups is Hamiltonian but does not have equationally definable principal congruences, it follows directly from Theorem~\ref{t:HamRL-charact} that, as already mentioned in Example~\ref{ex:ablggroupsnomodelcompletion} and first proved in~\cite{GP80}, its theory does not have a model completion. Similarly, the variety of MV-algebras does not have a model completion, as first proved in~\cite{Lac79}. \end{example}

An analogous result to Theorem~\ref{t:HamRL-charact} was obtained in~\cite{KM18} for Hamiltonian varieties of pointed residuated lattices that are closed under canonical extensions; for this latter notion, we refer to~\cite{GH01,Geh14}. 

 \begin{proposition}[{\cite[Theorem~5.11]{KM18}}]\label{p:coherence-edpc}
Let $\V$ be a Hamiltonian variety of pointed residuated lattices that is closed under canonical extensions. If $\V$ is coherent, then it has equationally definable principal congruences.
 \end{proposition}  

We use Proposition~\ref{p:coherence-edpc} to show that if a Hamiltonian semilinear variety $\V$ of pointed residuated lattices is closed under canonical extensions and the theory of $\Vtot$ has a model completion, then $\V$ has equationally definable principal congruences.

\begin{lemma}\label{l:varprojimplieseqvar}
Let $\V$ be a semilinear variety of pointed residuated lattices. If $\Vtot$ has the variable projection property, then $\V$ is coherent. 
\end{lemma}
\begin{proof}
The claim clearly holds if $\Vtot$ is trivial (contains only trivial algebras), so let us assume that this is not the case. Suppose that $\Vtot$ has the variable projection property and consider a finite set $\xbar,y$ and conjunction of equations $\f(\xbar,y)$. By assumption, there exists a quantifier-free formula $\xi(\xbar)$ such that $\Vtot\models\f\IMP\xi$ and for any equation $\ep(\xbar)$,
\[
\Vtot\models \f\IMP \ep \:\Longrightarrow\: \Vtot\models \xi\IMP\ep.
\]
We may assume without loss of generality that $\xi$ is a conjunction of formulas of the form $\pi \IMP \de$ where $\pi$ is a (possibly empty) conjunction of equations and $\de$ is a non-empty disjunction of equations. Using some elementary properties of pointed residuated lattices, we may also assume that $\de$ is of the form $\e\le s_1 \OR \cdots \OR \e \le s_n$. But also, using the fact that $\Vtot$ consists of linearly ordered pointed residuated lattices,
\[
\Vtot\models (\e\le s_1 \OR \cdots \OR \e \le s_n) \BIMP (\e \le s_1\jn\cdots\jn s_n).
\]
Hence we may further assume that $\xi$ is a conjunction of quasiequations. But then, since $\V$ and $\Vtot$ satisfy the same quasiequations, $\V$ also has the variable projection property and, by Proposition~\ref{p:coherence-eq-conseq}, is coherent.
\end{proof}

\begin{proposition}\label{p:failurelinearlyordered}
Let $\V$ be a Hamiltonian semilinear variety of pointed residuated lattices that is closed under canonical extensions. If the theory of $\Vtot$ has a model completion, then $\V$ has equationally definable principal congruences.
\end{proposition}  
\begin{proof}
Suppose that the theory of $\Vtot$ has a model completion. By Proposition~\ref{p:almost-coh-nec}, the class $\Vtot$ has the variable projection property. Hence, by Lemma~\ref{l:varprojimplieseqvar}, the variety $\V$ is coherent and, by Proposition~\ref{p:coherence-edpc}, has equationally definable principal congruences.
\end{proof}

\begin{example}
Consider the class $\cls{CRL^c}$ of linearly ordered commutative pointed residuated lattices. The variety generated by $\cls{CRL^c}$ is closed under canonical extensions (cf.~\cite[Chapter~6]{GJKO07}) and does not have equationally definable principal congruences. So, by Proposition~\ref{p:failurelinearlyordered}, the theory of $\cls{CRL^c}$ does not have a model completion.
\end{example}

Note that the Hamiltonian semilinear varieties $\LA$ and $\MV$ are coherent and do not have equationally definable principal congruences, but are not closed under canonical extensions. Moreover, as we have already seen, the theories of $\OG$ and $\MVC$ have a model completion, but this is not the case for  $\LA$ and $\MV$.

%%%%%%%%%%%%%%%%%%%%%%%%%%%%%%%%%%%

\section{Extending the language}\label{s:extension}

Let $\lang$ be the language of pointed residuated lattices and let $\langext$ be $\lang$ extended with an additional binary operation symbol $\trimp$. In this section, we show how to associate with any semilinear variety $\V$ of pointed residuated lattices, a variety $\Vext$ of $\langext$-algebras that has equationally definable principal congruences and satisfies the same universal $\lang$-sentences as $\V$. We then show that if $\V$ satisfies a certain syntactic property,  the theory of $\Vext$ has a model completion. In particular, this is the case for the varieties of lattice-ordered abelian groups and MV-algebras.

Given any semilinear variety $\V$ of pointed residuated lattices, let $\Vtotext$ denote the class of linearly ordered members of $\V$ expanded with a binary operation $\trimp$ defined by
\[
x \trimp y \coloneqq  
\begin{cases}
y & \text{if }\e \le x\\
\e & \text{otherwise.}
\end{cases}
\]
That is, $\Vtotext$ is the positive universal class consisting of $\langext$-algebras that satisfy the equational theory of $\V$ and the universal sentences
\[
\forall x,y \: (x\leq y \OR \, y\leq x)
\enspace\text{and}\enspace
\forall x,y \: [(\e\leq x \, \IMP \, x \trimp y\eq y) \AND (\e\not\leq x \, \IMP \, x \trimp y\eq \e) ].
\]
Let $\Vext$ be the variety generated by $\Vtotext$. Since $\Vtotext$ is a positive universal class, it follows from J\'{o}nsson's Lemma~\cite{Jon67} that $\Vext=\cop{ISP}(\Vtotext)$. Moreover, we obtain the following conservative extension result.

\begin{proposition}\label{p:vvt}
Let $\V$ be any semilinear variety of pointed residuated lattices. Then for any quantifier-free $\lang$-formula $\chi$,
\[
\Vext \models \chi\iff \V \models \chi.
\]
\end{proposition}
\begin{proof}
Since $\V$ and $\Vext$ are both varieties, they satisfy the disjunction property (see Section~\ref{s:algebraic-properties}), and it therefore suffices to establish the equivalence for the case where $\chi$ is an $\lang$-quasiequation. Suppose first that $\Vext \not\models \chi$. Since the $\lang$-reduct of any member of $\Vext$ belongs to $\V$, also $\V \not\models \chi$. Now suppose that $\V\not\models \chi$. Since $\V$ is semilinear, $\V=\cop{ISP}(\Vtot)$, and hence $\Vtot\not\models \chi$. But every member of $\Vtot$ is the $\lang$-reduct of a member of $\Vext$, and therefore also $\Vext \not\models \chi$. 
\end{proof}

For convenience of notation, let us define for any class $\K$ of algebras with a pointed residuated lattice reduct and a finite set of $\lang$-terms or $\langext$-terms $\Ga \cup \{t\}$,
\[
\Ga \mdl{\K} t \:\defiff \: \K\models \bigAND\{\e \le s \mid s \in \Ga\} \IMP \e \le t.
\]
It follows easily that for any conjunction of equations $\pi$ and equation $u\eq v$,
\[
\K\models\pi\IMP u\eq v \iff \{s\eqv t \mid s\eq t \text{ is an equation of } \pi\} \mdl{\K} u\eqv v.
\]
We now use this notation to describe a deduction theorem for $\Vext$.

\begin{proposition}\label{p:EDPCt}
Let $\V$ be any semilinear variety of pointed residuated lattices. Then for any finite set $\xbar$ and finite $\Ga \cup \{s,t\}\subseteq\Tmc_{\langext}(\xbar)$,
\[
\Ga\cup\{s\} \mdl{\Vext} t \iff \Ga \mdl{\Vext} s \trimp t.
\]
\end{proposition}
\begin{proof}
Using the fact that $\Vext=\cop{ISP}(\Vtotext)$, it suffices to prove that for any finite set $\xbar$ and finite $\Ga \cup \{s,t\}\subseteq\Tmc_{\langext}(\xbar)$,
\[
\Ga\cup\{s\} \mdl{\Vtotext} t \iff \Ga \mdl{\Vtotext} s \trimp t.
\]
Suppose first that $\Ga\cup\{s\} \mdl{\Vtotext} t$ and consider any $\m{A}\in\Vtotext$ and assignment $f\colon\xbar\to A$ satisfying $\e\le \tilde{f}(u)$ for all $u\in\Ga$. If $\tilde{f}(s)<\e$, then $\e =\tilde{f}(s)\trimp \tilde{f}(t) = \tilde{f}(s\trimp t)$; otherwise $\e\le \tilde{f}(s)$ and, by assumption, $\e\le \tilde{f}(t) = \tilde{f}(s)\trimp \tilde{f}(t) = \tilde{f}(s\trimp t)$. Hence $\Ga \mdl{\Vtotext} s \trimp t$. 

Suppose now that $\Ga \mdl{\Vtotext} s \trimp t$ and consider any $\m{A}\in\Vtotext$ and assignment $f\colon\xbar\to A$ satisfying $\e\le\tilde{f}(s)$ and $\e\le \tilde{f}(u)$ for all $u\in\Ga$. Then, by assumption, $\e\le \tilde{f}(s\trimp t)= \tilde{f}(s)\trimp \tilde{f}(t) = \tilde{f}(t)$. Hence $\Ga\cup\{s\} \mdl{\Vtotext} t$.
\end{proof}

The next result is then a direct consequence of Proposition~\ref{p:EDPC-reform}.

\begin{corollary}\label{c:EDPCt}
Let $\V$ be any semilinear variety of pointed residuated lattices. Then $\Vext$ has equationally definable principal congruences.
\end{corollary}

We also obtain a uniform method for transforming a disjunct in the conclusion of a consequence into a premise. For a finite set $\xbar = \{x_1,\dots,x_n\}$ and $s\in\Tmc_{\langext}(\xbar)$, let
\[
\trneg s \coloneqq s \trimp ((\zr\eqv\e)\mt\!\bigwedge_{1 \le j \le n} (x_j\eqv\e)).
\]

\begin{lemma}\label{l:reverse}
Let $\V$ be any semilinear variety of pointed residuated lattices. Then for any finite set $\xbar$ and finite $\Ga \cup \{s,t\}\subseteq\Tmc_{\langext}(\xbar)$, 
\[
\Ga \mdl{\Vext} s\jn t \iff \Ga\cup\{\trneg s\} \mdl{\Vext} t.
\]
\end{lemma}
\begin{proof}
Using the fact that $\Vext=\cop{ISP}(\Vtotext)$, it suffices to prove that for any finite set $\xbar = \{x_1,\dots,x_n\}$ and finite set $\Ga \cup \{s,t\}\subseteq\Tmc_{\langext}(\xbar)$,
\[
\Ga \mdl{\Vtotext} s\jn t \iff \Ga\cup\{\trneg s\} \mdl{\Vtotext} t.
\]
Suppose first that $\Ga \mdl{\Vtotext} s\jn t$ and consider any $\m{A}\in\Vtotext$ and assignment $f\colon\xbar\to A$ such that $\e\le \tilde{f}(\trneg s)$ and $\e\le \tilde{f}(u)$ for all $u\in\Ga$. Then, by assumption, $\e \le \tilde{f}(s\jn t)$. If $\e\le\tilde{f}(s)$, then $\e\le\tilde{f}(\trneg s)=\tilde{f}((\zr\eqv\e)\mt\bigwedge_{1 \le j \le n} (x_j\eqv \e))$, yielding $\zr =\tilde{f}(x_1)=\cdots = \tilde{f}(x_n)=\e$ and, inductively, $\tilde{f}(t)=\e$. Otherwise, $\e \le \tilde{f}(t)$.  Hence $\Ga\cup\{\trneg s\} \mdl{\Vtotext} t$. 

Suppose next that $\Ga\cup\{\trneg s\} \mdl{\Vtotext} t$ and  consider any $\m{A}\in\Vtotext$ and assignment $f\colon\xbar\to A$ such that $\e\le \tilde{f}(u)$ for all $u\in\Ga$. If $\e \le \tilde{f}(s)$, then $\e\le \tilde{f}(s\jn t)$. Otherwise, $\tilde{f}(\trneg s)=\e$ and, by assumption, $\e\le \tilde{f}(t)$, yielding $\e\le \tilde{f}(s\jn t)$. Hence  $\Ga \mdl{\Vtotext} s\jn t$.
\end{proof}

The next lemma provides a uniform method for eliminating occurrences of $\trimp$ from $\langext$-quasiequations while preserving validity in $\Vext$. To avoid multiple case distinctions, we introduce a new symbol $\tremp$, fixing $s\jn\tremp\coloneqq s$, $\tremp\jn s\coloneqq s$, $\trneg\tremp\coloneqq\e$ for any $\langext$-term $s$.

\begin{lemma}\label{l:removingdeltageneralt}
Let $\V$ be any semilinear variety of pointed residuated lattices. For any finite set $\xbar$, $s\in\Tmc_{\langext}(\xbar)$, and $t\in\Tmc_{\langext}(\xbar)\cup\{\tremp\}$, there exist a finite non-empty set $I$ and $s'_i\in\Tmc_{\lang}(\xbar)$, $t'_i\in\Tmc_{\lang}(\xbar)\cup\{\Lambda\}$ for each $i\in I$ such that for any $u,v \in\Tmc_{\langext}(\xbar,\ybar)$,
\[
\{u \mt s\}  \mdl{\Vext} t \jn v \iff \forall i\in I\colon\{u \mt s'_i\} \mdl{\Vext} t'_i \jn v.
\]
\end{lemma}
\begin{proof}
Let $s'[t']$ denote the result of replacing some distinguished occurrence of a variable in a term $s'$ by a term $t'$. For $s = s'[s_1\trimp s_2]$ and any $u,v \in\Tmc_{\langext}(\xbar,\ybar)$,
\[
\{u \mt s\}  \mdl{\Vext} t \jn v \iff \{u \mt s'[s_2] \mt s_1\}  \mdl{\Vext}  t \jn v  \:\text{ and }\: \{u \mt s'[\e]\} \mdl{\Vext}  t \jn s_1 \jn v.
\]
Similarly, for $t = t'[t_1\trimp t_2]$  and any $u,v \in\Tmc_{\langext}(\xbar,\ybar)$,
\[
\{u \mt s\} \mdl{\Vext} t \jn v \iff \{u \mt s \mt t_1\}  \mdl{\Vext}  t'[t_2] \jn v \:\text{ and }\:\{u \mt s\} \mdl{\Vext}  t'[\e] \jn t_1 \jn v.
\]
Iterating these steps yields the required finite set $I$ and $s'_i,t'_i\in\Tmc_{\lang}(\xbar)$ for each $i\in I$.
\end{proof}

We now provide sufficient conditions for a semilinear variety of pointed residuated lattices $\V$ to ensure that the theory of $\Vext$ has a model completion. 

\begin{theorem}\label{t:two-sided-mc}
Let $\V$ be any semilinear variety of pointed residuated lattices such that for any finite set $\xbar,y$ and $s(\xbar,y)\in\Tmc_{\lang}(\xbar,y)$, $t(\xbar,y)\in\Tmc_{\lang}(\xbar,y)\cup\{\tremp\}$, there exist a finite non-empty set $K$ and $s'_k(\xbar)\in\Tmc_{\lang}(\xbar)$, $t'_k(\xbar)\in\Tmc_{\lang}(\xbar)\cup\{\tremp\}$ for each $k\in K$ satisfying for any $u(\xbar,\zbar),v(\xbar,\zbar)\in\Tmc_{\lang}(\xbar,\zbar)$,
\[
\{u(\xbar,\zbar)\mt s(\xbar,y)\} \mdl{\V} t(\xbar,y)\jn v(\xbar,\zbar) \iff \forall k\in K\colon\{u(\xbar,\zbar)\mt s'_k(\xbar)\} \mdl{\V} t'_k(\xbar)\jn v(\xbar,\zbar).
\]
Then the theory of $\Vext$ has a model completion. 
\end{theorem}
\begin{proof}
From Corollary~\ref{c:EDPCt}, we know that $\Vext$ has equationally definable principal congruences.  Hence, to conclude using Proposition~\ref{p:suffconditionsedpc} that the theory of $\Vext$ has a model completion, it remains to prove that $\Vext$ is coherent and has the amalgamation property and equational variable restriction property. 

For coherence, it suffices to show that for any finite set $\xbar,y$ and $s(\xbar,y)\in\Tmc_{\langext}(\xbar,y)$, there exists an $s^\star(\xbar)\in\Tmc_{\langext}(\xbar)$ such that for any $v(\xbar) \in\Tmc_{\langext}(\xbar)$, 
\[
\{s(\xbar,y)\} \mdl{\Vext} v(\xbar) \iff \{s^\star(\xbar)\} \mdl{\Vext} v(\xbar).
\]
By Lemma~\ref{l:removingdeltageneralt} (with $t=\tremp$), there exist a finite non-empty set $I$ and $s'_i(\xbar,y)\in\Tmc_{\lang}(\xbar,y)$, $t'_i(\xbar,y)\in\Tmc_{\lang}(\xbar,y)\cup\{\tremp\}$ ($i\in I$) such that for any $v(\xbar) \in\Tmc_{\langext}(\xbar)$, 
\begin{align}\label{equivs_for_alpha_and_beta}
\{s(\xbar,y)\} \mdl{\Vext} v(\xbar) & \iff \forall i\in I\colon\{s'_i(\xbar,y)\}  \mdl{\Vext} t'_i(\xbar,y)\jn v(\xbar).
\end{align}
By assumption, there exist for each $i\in I$, a finite non-empty set $K_i$ and $s''_{i,k}(\xbar)\in\Tmc_{\lang}(\xbar)$, $t''_{i,k}(\xbar)\in\Tmc_{\lang}(\xbar)\cup\{\tremp\}$ ($k\in K_i$) satisfying for any $u'(\xbar),v'(\xbar)\in\Tmc_{\lang}(\xbar)$,
\begin{align}\label{interpolants}
\begin{split}
 \{u'(\xbar)\mt s'_i(\xbar,y)\} & \mdl{\V} t'_i(\xbar,y)\jn v'(\xbar)\\
 & \iff \forall k\in K_i\colon \{u'(\xbar)\mt s''_{i,k}(\xbar)\} \mdl{\V} t''_{i,k}(\xbar)\jn v'(\xbar).
 \end{split}
\end{align}
We define now
\[
s^\star(\xbar)\coloneqq \bigvee_{i\in I}\bigvee_{k\in K_i}\left(s''_{i,k}(\xbar)  \mt  \trneg t''_{i,k}(\xbar)\right).
\]
Consider any $v(\xbar)\in\Tmc_{\langext}(\xbar)$. By Lemma~\ref{l:removingdeltageneralt} (with $s=\e$), there exist a finite non-empty set $J$ and $u'_j(\xbar)\in\Tmc_{\lang}(\xbar)$, $v'_j(\xbar)\in\Tmc_{\lang}(\xbar)\cup\{\Lambda\}$ ($j\in J$) such that for any $s'(\xbar,y),t'(\xbar,y)\in\Tmc_{\langext}(\xbar,y)$,
\[
\{s'(\xbar,y)\}  \mdl{\Vext} t'(\xbar,y)\jn v(\xbar) \iff \forall j\in J\colon\{u'_j(\xbar)\mt s'(\xbar,y)\} \mdl{\Vext}t'(\xbar,y)\jn v'_j(\xbar).
\]
In particular, for each $i \in I$ and $k\in K_i$,
\begin{align}
\label{equiv_one}
\{s'_i(\xbar,y)\}  \mdl{\Vext} t'_i(\xbar,y)\jn v(\xbar)  & \iff  \forall j\in J\colon\{u'_j(\xbar)\mt s'_i(\xbar,y)\}  \mdl{\Vext} t'_i(\xbar,y)\jn v'_j(\xbar)\\
\label{equiv_two}
\{s''_{i,k}(\xbar)\}  \mdl{\Vext} t''_{i,k}(\xbar)\jn v(\xbar) & \iff  \forall j\in J\colon\{u'_j(\xbar)\mt s''_{i,k}(\xbar)\}  \mdl{\Vext} t''_{i,k}(\xbar)\jn v'_j(\xbar).
\end{align}
Putting these equivalences together, it follows that
\begin{align*}
\{s(\xbar,y)\} & \mdl{\Vext} v(\xbar)\\
& \iff \forall i\in I\colon\{s'_i(\xbar,y)\}  \mdl{\Vext} t'_i(\xbar,y) \jn v(\xbar) 
& \text{by~\eqref{equivs_for_alpha_and_beta}}\\
& \iff \forall i\in I,j\in J\colon\{u'_j(\xbar)\mt s'_i(\xbar,y)\}  \mdl{\Vext} t'_i(\xbar,y)\jn v'_j(\xbar) 
& \text{by~\eqref{equiv_one}}\\
& \iff \forall i\in I,j\in J\colon\{u'_j(\xbar)\mt s'_i(\xbar,y)\}  \mdl{\V} t'_i(\xbar,y)\jn v'_j(\xbar)
& \text{Prop.~\ref{p:vvt}}\\
& \iff \forall i\in I,j\in J,k\in K_i\colon\{u'_j(\xbar)\mt s''_{i,k}(\xbar)\} \mdl{\V} t''_{i,k}(\xbar)\jn v'_j(\xbar)
& \text{by~\eqref{interpolants}}\\
& \iff \forall i\in I,j\in J,k\in K_i\colon\{u'_j(\xbar)\mt s''_{i,k}(\xbar)\} \mdl{\Vext} t''_{i,k}(\xbar)\jn v'_j(\xbar)
& \text{Prop.~\ref{p:vvt}}\\
& \iff \forall i\in I,k\in K_i\colon\{s''_{i,k}(\xbar)\} \mdl{\Vext} t''_{i,k}(\xbar)\jn v(\xbar)
& \text{by~\eqref{equiv_two}}\\
& \iff \forall i\in I,k\in K_i\colon\{s''_{i,k}(\xbar)\mt \trneg t''_{i,k}(\xbar)\} \mdl{\Vext}  v(\xbar) 
& \text{Lem.~\ref{l:reverse}}\\
& \iff \{s^\star(\xbar)\} \mdl{\Vext} v(\xbar). 
\end{align*}
For the amalgamation property and equational variable restriction property, it suffices by Corollary~\ref{c:amalg+leftvar} to show that for any finite set $\xbar,y$ and $t(\xbar,y)\in\Tmc_{\langext}(\xbar,y)$, either $\{w(\xbar)\}\nmdl{\Vext} t(\xbar,y)$ for all $w(\xbar)\in\Tmc_{\langext}(\xbar)$ (taking care of the case in Corollary~\ref{c:amalg+leftvar} where $\pi(\xbar)$ is $\bot$) or there exists a $t^\star(\xbar)\in\Tmc_{\langext}(\xbar)$ such that for any $u(\xbar,\zbar) \in\Tmc_{\langext}(\xbar,\zbar)$, 
\[
\{u(\xbar,\zbar)\} \mdl{\Vext} t(\xbar,y) \iff \{u(\xbar,\zbar)\} \mdl{\Vext} t^\star(\xbar).
\]
Suppose then that $\{w(\xbar)\} \mdl{\Vext} t(\xbar,y)$ for some $w(\xbar)\in\Tmc_{\langext}(\xbar)$. By iteratively replacing $w'[w_1\trimp w_2]$ with $w'[w_2]\mt w_1$, we may assume without loss of generality that  $w(\xbar)\in\Tmc_{\lang}(\xbar)$. Next, by Lemma~\ref{l:removingdeltageneralt} (with $s=\e$), there exist a finite  non-empty set $I$ and $s'_i(\xbar,y),t'_i(\xbar,y)\in\Tmc_{\lang}(\xbar,y)$ ($i\in I$) such that for any $u(\xbar,\zbar)\in\Tmc_{\langext}(\xbar,\zbar)$, 
\begin{align}
\{u(\xbar,\zbar)\}  \mdl{\Vext} t(\xbar,y) & \iff \forall i\in I\colon\{u(\xbar,\zbar)\mt s'_i(\xbar,y)\} \mdl{\Vext} t'_i(\xbar,y). \label{equivs_for_alpha_and_betat}
\end{align}
Let us temporarily fix $i\in I$. By assumption, there exist a finite non-empty set $K_i$ and $s''_{i,k}(\xbar)\in\Tmc_{\lang}(\xbar)$, $t''_{i,k}(\xbar)\in\Tmc_{\lang}(\xbar)\cup\{\tremp\}$ ($k\in K_i$) such that for any $u'(\xbar,\zbar),v'(\xbar,\zbar)\in\Tmc_{\lang}(\xbar,\zbar)$,
\begin{align}\label{interpolantst}
\begin{split}
\{u'(\xbar,\zbar)\mt s'_i(\xbar,y)\} & \mdl{\V} t'_i(\xbar,y)\jn v'(\xbar,\zbar)\\
 & \iff \forall k\in K_i\colon \{u'(\xbar,\zbar)\mt s''_{i,k}(\xbar)\} \mdl{\V} t''_{i,k}(\xbar)\jn v'(\xbar,\zbar).
 \end{split}
\end{align}
Since $\{w(\xbar)\} \mdl{\Vext} t(\xbar,y)$, by~\eqref{equivs_for_alpha_and_betat} and Proposition~\ref{p:vvt}, also $\{w(\xbar)\mt s'_i(\xbar,y)\} \mdl{\V} t'_i(\xbar,y)$. Let $t'''_{i,k}(\xbar) \coloneqq  t''_{i,k}(\xbar) \jn w(\xbar)$  for each $k\in K_i$. Then for any $u'(\xbar,\zbar),v'(\xbar,\zbar)\in\Tmc_{\lang}(\xbar,\zbar)$,
\begin{align}\label{interpolantstnew}
\begin{split}
\{u'(\xbar,\zbar)\mt s'_i(\xbar,y)\} & \mdl{\Vext} t'_i(\xbar,y)\jn v'(\xbar,\zbar)\\
 & \iff \forall k\in K_i\colon\{u'(\xbar,\zbar)\mt s''_{i,k}(\xbar)\} \mdl{\Vext} t'''_{i,k}(\xbar)\jn v'(\xbar,\zbar).
 \end{split}
\end{align}
The left-to-right direction of this equivalence is immediate. For the converse direction, 
\begin{align*}
& \forall k\in K_i\colon\{u'(\xbar,\zbar)\mt s''_{i,k}(\xbar)\} \mdl{\Vext} t'''_{i,k}(\xbar)\jn v'(\xbar,\zbar)\\
&  \Longrightarrow  \forall k\in K_i\colon\{u'(\xbar,\zbar)\mt s''_{i,k}(\xbar)\} \mdl{\V} t'''_{i,k}(\xbar)\jn v'(\xbar,\zbar) & \text{Prop.~\ref{p:vvt}}\\
& \Longrightarrow \{u'(\xbar,\zbar)\mt s'_i(\xbar,y)\}\mdl{\V} t'_i(\xbar,y)\jn w(\xbar) \jn v'(\xbar,\zbar) & \text{by~\eqref{interpolantst}}\\
& \Longrightarrow \{u'(\xbar,\zbar)\mt s'_i(\xbar,y)\}\mdl{\V} t'_i(\xbar,y)\jn v'(\xbar,\zbar) & \text{since $\{w(\xbar)\mt s'_i(\xbar,y)\} \mdl{\V} t'_i(\xbar,y)$}\\
& \Longrightarrow \{u'(\xbar,\zbar)\mt s'_i(\xbar,y)\}\mdl{\Vext} t'_i(\xbar,y)\jn v'(\xbar,\zbar)& \text{Prop.~\ref{p:vvt}}.
\end{align*}
We define now
\[
t^\star(\xbar)\coloneqq \bigwedge_{i\in I}\bigwedge_{k\in K_i} \left(s''_{i,k}(\xbar) \trimp t'''_{i,k}(\xbar)\right).
\]
Consider any $u(\xbar,\zbar)\in\Tmc_{\langext}(\xbar,\zbar)$. By Lemma~\ref{l:removingdeltageneralt} (with $t=\tremp$), there exist a finite non-empty set $J$ and $u'_j(\xbar,\zbar)\in\Tmc_{\lang}(\xbar,\zbar)$, $v'_j(\xbar,\zbar)\in\Tmc_{\lang}(\xbar,\zbar)\cup\{\tremp\}$ ($j\in J$) such that for any $s'(\xbar,y),t'(\xbar,y)\in\Tmc_{\langext}(\xbar,y)$,
\[
\{u(\xbar,\zbar)\mt s'(\xbar,y)\}  \mdl{\Vext} t'(\xbar,y) \iff \forall j\in J\colon\{u'_j(\xbar,\zbar)\mt s'(\xbar,y)\} \mdl{\Vext}t'(\xbar,y)\jn v'_j(\xbar,\zbar).
\]
In particular, for each $i \in I$ and $k\in K_i$,
\begin{align}
\label{equivs_onet}
\begin{split}
\{u(\xbar,\zbar) \mt s'_i(\xbar,y)\} & \mdl{\Vext} t'_i(\xbar,y)\\
& \iff \forall j\in J\colon\{u'_j(\xbar,\zbar) \mt s'_i(\xbar,y)\}  \mdl{\Vext} t'_i(\xbar,y)\jn v'_j(\xbar,\zbar)
\end{split}
\end{align}
\begin{align}
\label{equivs_twot}
\begin{split}
\{u(\xbar,\zbar) \mt s''_{i,k}(\xbar)\} & \mdl{\Vext} t'''_{i,k}(\xbar)\\
& \iff \forall j\in J\colon\{u'_j(\xbar,\zbar) \mt s''_{i,k}(\xbar)\}  \mdl{\Vext} t'''_{i,k}(\xbar)\jn v'_j(\xbar,\zbar).
\end{split}
\end{align}
Putting these equivalences together, it follows that 
\begin{align*}
\{u(\xbar,\zbar)\} & \mdl{\Vext} t(\xbar,y)\\
& \iff \forall i\in I\colon\{u(\xbar,\zbar) \mt s'_i(\xbar,y)\}  \mdl{\Vext} t'_i(\xbar,y)
& \text{by~\eqref{equivs_for_alpha_and_betat}}\\
& \iff \forall i\in I,j\in J\colon\{u'_j(\xbar,\zbar) \mt s'_i(\xbar,y)\}  \mdl{\Vext} t'_i(\xbar,y)\jn v'_j(\xbar,\zbar)
& \text{by~\eqref{equivs_onet}}\\
& \iff \forall i\in I,j\in J,k\in K_i\colon\{u'_j(\xbar,\zbar)  \mt s''_{i,k}(\xbar)\}  \mdl{\Vext} t'''_{i,k}(\xbar)\jn v'_j(\xbar,\zbar) 
& \text{by~\eqref{interpolantstnew}}\\
& \iff \forall i\in I,k\in K_i\colon\{u(\xbar,\zbar) \mt s''_{i,k}(\xbar)\} \mdl{\Vext} t'''_{i,k}(\xbar) 
& \text{by~\eqref{equivs_twot}}\\
& \iff \forall i\in I,k\in K_i\colon\{u(\xbar,\zbar)\} \mdl{\Vext}  s''_{i,k}(\xbar) \trimp t'''_{i,k}(\xbar)
& \text{Prop.~\ref{p:EDPCt}}\\
& \iff \{u(\xbar,\zbar)\}  \mdl{\Vext} t^\star(\xbar). & \qed\proofend
\end{align*}
\end{proof}

In particular, the conditions of Theorem~\ref{t:two-sided-mc} are satisfied when $\V$ is the variety $\LA$ of lattice-ordered abelian groups.

\begin{theorem}\label{t:laext}
The theory of $\LAext$ has a model completion.
\end{theorem}\begin{proof}
Recall that $\LA$ is generated as a quasivariety by the lattice-ordered abelian group $\m{R}=\lan\R,\min,\max,+,-,0\ran$ (see, e.g.,~\cite{AF88}). Hence, to show that $\LAext$ has a model completion, it suffices to check the condition of Theorem~\ref{t:two-sided-mc} with $\mdl{\V}$ replaced by $\mdl{\m{R}}$. 

We proceed in two stages. Suppose first that $\xbar,y$ is a finite set, $s(\xbar,y)\in\Tmc_{\lang}(\xbar,y)$ is a meet of group terms, and $t(\xbar,y)\in\Tmc_{\lang}(\xbar,y)\cup\{\tremp\}$ is $\tremp$ or a join of group terms. We prove that there exist $s'(\xbar)\in\Tmc_{\lang}(\xbar)$, $t'(\xbar)\in\Tmc_{\lang}(\xbar)\cup\{\tremp\}$ satisfying for any $u(\xbar,\zbar),v(\xbar,\zbar)\in\Tmc_{\lang}(\xbar,\zbar)$,
\[
\{u(\xbar,\zbar)\mt s(\xbar,y)\} \mdl{\m{R}} t(\xbar,y)\jn v(\xbar,\zbar)\iff \{u(\xbar,\zbar)\mt s'(\xbar)\} \mdl{\m{R}} t'(\xbar)\jn v(\xbar,\zbar).
\]
It is easily checked in $\m{R}$ that for any $\lang$-terms $u,v,w$ and $k\in\N^{>0}$,
\[
\{u \mt v\} \mdl{\m{R}} w \iff \{u \mt kv\} \mdl{\m{R}} w
\enspace\text{and}\enspace
\{u\} \mdl{\m{R}} v \jn w \iff \{u\} \mdl{\m{R}} v \jn kw.
\]
Also, $\m{R}\models x + (y \mt z) \eq (x+y)\mt (x+z)$ and  $\m{R}\models x + (y \jn z) \eq (x+y)\jn (x+z)$. Hence, reasoning in $\m{R}$, we may assume that there exist a $k\in\N^{>0}$ and $s_0(\xbar),s_1(\xbar),s_2(\xbar),t_1(\xbar),t_2(\xbar)\in\Tmc_{\lang}(\xbar)$, $t_0(\xbar)\in\Tmc_{\lang}(\xbar)\cup\{\tremp\}$ such that $s(\xbar,y)$ is the meet of $s_0(\xbar)$ and some members (possibly none) of $\{s_1(\xbar) +ky, s_2(\xbar)-ky\}$, and $t(\xbar,y)$ is the join of $t_0(\xbar)$ and some members (possibly none) of $\{t_1(\xbar) + ky,t_2(\xbar)-ky\}$. We construct $s'(\xbar)\in\Tmc_{\lang}(\xbar)$, $t'(\xbar)\in\Tmc_{\lang}(\xbar)\cup\{\tremp\}$ as follows. Let  $s'(\xbar)\coloneqq  s_0(\xbar) \mt (s_1(\xbar)+s_2(\xbar))$ if $s_1(\xbar) +ky$ and $s_2(\xbar)-ky$ occur in $s(\xbar,y)$; otherwise $s'(\xbar)\coloneqq s_0(\xbar)$. The term $t'(\xbar)$ is the join of $t_0(\xbar)$ and (i) $t_1(\xbar)+t_2(\xbar)$ if $t_1(\xbar) + ky$ and $t_2(\xbar)-ky$ occur in $t(\xbar,y)$; (ii) $t_1(\xbar)-s_1(\xbar)$ if $s_1(\xbar) +ky$ occurs in $s(\xbar,y)$ and $t_1(\xbar) + ky$ occurs in $t(\xbar,y)$; (iii) $t_2(\xbar)-s_2(\xbar)$ if $s_2(\xbar)-ky$ occurs in $s(\xbar,y)$ and $t_2(\xbar) - ky$ occurs in $t(\xbar,y)$. 

We check the case where $s(\xbar,y)= (s_1(\xbar) +ky)\mt (s_2(\xbar)-ky)$, $t(\xbar,y) = (t_1(\xbar) + ky)\jn (t_2(\xbar)-ky)$, $s'(\xbar) = s_1(\xbar)+s_2(\xbar)$, and $t'(\xbar) = (t_1(\xbar)+t_2(\xbar))\jn(t_1(\xbar)-s_1(\xbar))\jn(t_2(\xbar)-s_2(\xbar))$, other cases being very similar. Let $u(\xbar,\zbar),v(\xbar,\zbar)\in\Tmc_{\lang}(\xbar,\zbar)$ and suppose first that $\{u(\xbar,\zbar)\mt s'(\xbar)\} \mdl{\m{R}} t'(\xbar)\jn v(\xbar,\zbar)$. Since $\{s(\xbar,y)\}\mdl{\m{R}}s'(\xbar)$ and $\{s(\xbar)\mt t'(\xbar)\}\mdl{\m{R}}t(\xbar,y)$, it follows that $\{u(\xbar,\zbar)\mt s(\xbar,y)\} \mdl{\m{R}} t(\xbar,y)\jn v(\xbar,\zbar)$. Now suppose that $\{u(\xbar,\zbar)\mt s'(\xbar)\}\nmdl{\m{R}}t'(\xbar)\jn v(\xbar,\zbar)$; that is, some assignment $f\colon\xbar,\zbar\to \R$ satisfies $0\le \tilde{f}(u(\xbar,\zbar))$,  $0\le \tilde{f}(s_1(\xbar))+ \tilde{f}(s_2(\xbar))$, $\tilde{f}(t_1(\xbar))+\tilde{f}(t_2(\xbar))<0$,  $\tilde{f}(t_1(\xbar))<\tilde{f}(s_1(\xbar))$, $\tilde{f}(t_2(\xbar))<\tilde{f}(s_2(\xbar))$, and $\tilde{f}(v(\xbar,\zbar))<0$. We extend $f$ to an assignment $g\colon\xbar,y,\zbar\to \R$ by defining
\[
g(y)\coloneqq\frac{\min{(\tilde{f}(s_2),-\tilde{f}(t_1))}+\max{(-\tilde{f}(s_1),\tilde{f}(t_2))}}{2k}.
\]
We obtain $0\le\tilde{g}(u(\xbar,\zbar))$, $0\le \tilde{g}(s_1(\xbar))+\tilde{g}(ky)$, $\tilde{g}(ky)\le \tilde{g}(s_2(\xbar))$, $\tilde{g}(t_1(\xbar)) + \tilde{g}(ky) <0$, $\tilde{g}(t_2(\xbar))<\tilde{g}(ky)$, and $\tilde{g}(v(\xbar,\zbar))<0$. So $\{u(\xbar,\zbar)\mt s(\xbar,y)\}\nmdl{\m{R}}t(\xbar,y)\jn v(\xbar,\zbar)$.

We now prove that for any finite set $\xbar,y$ and $s(\xbar,y)\in\Tmc_{\lang}(\xbar,y)$, $t(\xbar,y)\in\Tmc_{\lang}(\xbar,y)\cup\{\tremp\}$, there exist a finite non-empty set $K$ and $s'_k(\xbar)\in\Tmc_{\lang}(\xbar)$, $t'_k(\xbar)\in\Tmc_{\lang}(\xbar)\cup\{\tremp\}$ ($k\in K$) satisfying for any $u(\xbar,\zbar),v(\xbar,\zbar)\in\Tmc_{\lang}(\xbar,\zbar)$,
\[
\{u(\xbar,\zbar)\mt s(\xbar,y)\} \mdl{\m{R}} t(\xbar,y)\jn v(\xbar,\zbar) \iff \forall k\in K\colon\{u(\xbar,\zbar)\mt s'_k(\xbar)\} \mdl{\m{R}} t'_k(\xbar)\jn v(\xbar,\zbar).
\]
By the distributivity properties of $\m{R}$, we may assume that $s(\xbar,y)$ is $s_1(\xbar,y)\jn\cdots\jn s_n(\xbar,y)$, where each $s_i(\xbar,y)$ is a meet of group terms, and that $t(\xbar,y)$ is $t_1(\xbar,y)\mt\cdots\mt t_m(\xbar,y)$, where either each $t_j(\xbar,y)$ is a join of group terms or $m=1$ and $t_1(\xbar,y)=\tremp$. Using the first part of this proof, for each $i\in\{1,\dots,n\}$ and $j\in\{1,\dots,m\}$, there exist $s'_{i,j}(\xbar)\in\Tmc_{\lang}(\xbar)$ and $t'_{i,j}(\xbar)\in\Tmc_{\lang}(\xbar)\cup\{\tremp\}$ satisfying for any $u(\xbar,\zbar),v(\xbar,\zbar)\in\Tmc_{\lang}(\xbar,\zbar)$, 
\[
\{u(\xbar,\zbar)\mt s_i(\xbar,y)\} \mdl{\m{R}} t_j(\xbar,y)\jn v(\xbar,\zbar)\iff \{u(\xbar,\zbar)\mt s'_{i,j}(\xbar)\} \mdl{\m{R}} t'_{i,j}(\xbar)\jn v(\xbar,\zbar).
\]
Hence for any $u(\xbar,\zbar),v(\xbar,\zbar)\in\Tmc_{\lang}(\xbar,\zbar)$,
\begin{align*}
\{u(\xbar,\zbar)\mt s(\xbar,y)\} \mdl{\m{R}} t(\xbar,y)\jn v(\xbar,\zbar) \iff & \forall i\in\{1,\dots,n\},j\in\{1,\dots,m\}\colon\\
&  \{u(\xbar,\zbar)\mt s'_{i,j}(\xbar)\} \mdl{\m{R}} t'_{i,j}(\xbar)\jn v(\xbar,\zbar).& \qed\proofend
\end{align*}
\end{proof}

The variety $\MVext$ of MV-algebras extended with $\trimp$ is term-equivalent to the variety $\MVD$ of {\em MV$_\De$-algebras} (see, e.g.,~\cite{Haj98,MP01}), generated by linearly ordered MV-algebras extended with an additional unary operator $\De$ defined by
\[
\De x \coloneqq  
\begin{cases}
1 & \text{if }x=1\\
0 & \text{otherwise.}
\end{cases}
\]
For any algebra in $\MVext$, we let $\De x\coloneqq  (x\trimp 0)\trimp 0$, and for any algebra in $\MVD$, we let $x\trimp y\coloneqq  \lnot\De x \oplus y$. The following theorem may therefore be proved along the same lines as Theorem~\ref{t:laext}, using the fact that $\MV$ is generated as a quasivariety by the algebra $\langle [0,1],\oplus,\lnot,0\rangle$, where $x\oplus y\coloneqq \min{(1,x+y)}$ and $\lnot x\coloneqq1-x$ (see,~e.g.,~\cite{COM99}).

\begin{theorem}
The theory of $\MVD$ has a model completion.
\end{theorem}

A different proof of this last result (obtained by giving an explicit axiomatization) was presented by X.~Caicedo at the conference {\em Residuated Structures: Algebra and Logic} in 2008.  Caicedo also presented a proof (again by giving an explicit axiomatization) at the {\em Latin American Algebra Colloquium} in  2019 that the theory of the class of so-called ``pseudo-complemented'' lattice-ordered abelian groups has a model completion, which seems to bear some relation to our Theorem~\ref{t:laext}.

\subsection*{Acknowledgement} 
The authors would like to thank the anonymous referee for their careful reading of the paper and valuable comments.

%%%%%%%%%%%%%%%%%%%%%%%%%%%%%%%%%%%%%%%%%%

\appendix

\section{A comparison with Wheeler's characterization theorem}\label{app:wheeler}

In this appendix, we relate our results in Section~\ref{s:model-completions} to the necessary and sufficient conditions provided by Wheeler in~\cite{Whe76} (see also~\cite{Whe78}) for the existence of model completions for universal theories with finite presentations. Let us note that Wheeler considers arbitrary first-order languages and, although we restrict ourselves here to algebraic languages, it is not difficult to see that the results of this appendix can be generalized to any first-order language by replacing equations with atomic formulas.

Let $\K$ be a universal class of algebras. Following~\cite[Section~3]{Whe76}, a \emph{finite presentation} of an algebra $\m A\in\K$ is a pair $(\abar, \pi)$ where $\abar\in A^{\xbar}$ is a finite set of generators for $\m A$ and $\pi(\xbar)$ is a conjunction of equations such that $\m A\models \pi(\abar)$ and for any equation $\ep(\xbar)$,
\[
\m A\models \ep(\abar) \iff \K\models \pi\IMP\ep.
\] 
Observe that $(\abar, \pi)$ is a finite presentation of $\m A$ if, and only if, for any $\m B\in\K$ generated by $\bbar\in B^{\xbar}$ satisfying $\m B\models \pi(\bbar)$, there is a surjective homomorphism $\m A\onto \m B$ mapping $\abar$ to~$\bbar$. An algebra $\m A\in\K$ is said to be \emph{finitely presented} in $\K$ if it admits a finite presentation.

\begin{remark}\label{r:finpres}
If $\K$ is a variety, then an algebra $\m A \in \K$ is finitely presented in $\K$ if, and only if, it is finitely presented in the usual sense, i.e., $\m{A}$ is isomorphic to the quotient of a finitely generated $\K$-free algebra with respect to a compact congruence. Moreover, if $\K$ is a positive universal class and the variety $\V$ generated by $\K$ is congruence distributive, then $\m A\in \K$ is finitely presented in $\K$ if, and only if, it is finitely presented in $\V$. Just observe that by J\'{o}nsson's Lemma~\cite{Jon67}, we have in this case $\V=\cop{ISP}(\K)$.
\end{remark}

A class of algebras $\K$ has \emph{finite presentations} if for every finite set $\xbar$ and conjunction of equations $\pi(\xbar)$ that is satisfiable in $\K$, there exist an algebra $\m A\in \K$ and a tuple $\abar\in A^{\xbar}$ such that $(\abar,\pi)$ is a presentation of $\m A$. All quasivarieties, and in particular all varieties, have finite presentations (cf.~\cite[Corollary~1, p.~315]{Whe76}), but this is not the case for all universal classes, as shown by the following example.

\begin{example}
The positive universal class $\OG$ of ordered abelian groups does not have finite presentations. To see this, recall that $\OG$ generates the variety $\LA$ of lattice-ordered abelian groups, which is congruence distributive (cf.~Example~\ref{ex:ordered-groups-alm-coh}). Hence, by Remark~\ref{r:finpres}, an ordered abelian group is finitely presented in the sense of the above definition if, and only if, it is finitely presented as a lattice-ordered abelian group in the usual sense. However, any finitely presented ordered abelian group is simple (see, e.g.,~\cite[Theorem~4.A and Corollary~5.2.3]{Glass1999}) and there exist finitely generated ordered abelian groups that are not simple, e.g., the lexicographic product $\m{R}\overrightarrow{\times}\m{R}$ generated by $\{(1,0),(0,1)\}$. It follows that there cannot be a $2$-generated finitely presented ordered abelian group with presentation $\top(x_1,x_2)$.
\end{example}

Let us now recall Wheeler's conservative congruence extension property. Given any algebra $\m A\in\K$, let $\Diag(\m A)$ be the positive diagram of $\m A$, i.e., the set of atomic sentences in the language extended with names for the elements of $\m A$ that are satisfied in $\m A$. The class $\K$ has the \emph{conservative congruence extension property (for finite presentations)} if, whenever $\m B$ admits a finite presentation $((\abar,\bbar),\pi(\xbar,\ybar))$, $\m A$ is the subalgebra of $\m B$ generated by $\abar$, the tuple $\bbar$ does not lie in $\m A$, and $\rho(\xbar,\ybar)$ is a conjunction of negated equations such that $\m B \models \rho(\abar,\bbar)$, there exists a quantifier-free formula $\chi(\xbar)$ satisfying 
\[
\Th(\K)\cup\Diag(\m B)\vdash \rho(\abar,\bbar)\IMP \chi(\abar),
\] 
and for every surjective homomorphism $h\colon \m A\onto \m A'$ such that $\m A'\in \K$ and $\m A'\models \chi(h(\abar))$ there exist a $\m B'\in \K$ extending $\m A'$ and a surjective homomorphism $h'\colon \m B \onto \m B'$ whose restriction to $\m A$ coincides with $h$, and such that $\m B'\models \rho(h(\abar),h'(\bbar))$.\footnote{Note that Wheeler does not assume that $\m A'\in \K$, but uses this in the proof of his main result.}

The following proposition shows that, for universal classes with finite presentations, the conservative congruence extension property is equivalent to a strengthening of the conservative model extension property where the variable $y$ is replaced by a tuple $\ybar$.

\begin{proposition}\label{p:CMEP-CCEP}
Let $\K$ be a universal class of algebras with finite presentations. Then $\K$ has the conservative congruence extension property if, and only if, for any finite sets $\xbar,\ybar$ and conjunction of literals $\ps(\xbar,\ybar)$, there exists a quantifier-free formula $\chi(\xbar)$ satisfying 
 \begin{enumerate}
 \item [(i)] $\K\models \ps\IMP\chi$
 \item [(ii)] for every  $\m A \in \K$ generated by $\abar\in A^{\xbar}$ such that $\m A\models\chi(\abar)$ and for any equation $\ep(\xbar)$,
 \[
 \K\models\ps^+\IMP\ep \:\Longrightarrow\: \m A\models\ep(\abar),
 \]
 there exist an algebra $\m B\in\K$ extending $\m A$ and $\bbar \in B^{\ybar}$ such that $\m B\models \ps(\abar,\bbar)$.
 \end{enumerate}
\end{proposition}
\begin{proof}
Suppose that $\K$ has the conservative congruence extension property. Fix finite sets $\xbar,\ybar$ and a conjunction of literals $\ps(\xbar,\ybar)$. If $\ps$ is not satisfiable in $\K$, we can set $\chi\coloneqq \bot$. Hence, assume that $\ps$ is satisfiable in $\K$. Consider $\m B\in \K$ and $(\abar,\bbar)\in B^{\xbar,\ybar}$ such that $((\abar,\bbar),\ps^+)$ is a finite presentation of $\m B$ and let $\m A$ be the subalgebra of $\m B$ generated by $\abar$. If there is a $b_i\in \bbar$ such that $b_i\in A$, then $\K\models \ps^+ \IMP y_i\eq t(\xbar)$ for some term~$t$. Replacing $y_i$ by $t(\xbar)$ in the formula $\ps$ whenever $b_i\in A$, we can assume that no element of $\bbar$ belongs to $\m A$. Further, if $\m B\models \si(\abar,\bbar)$ for some equation $\si$ of $\ps^-$, then $\K\models \ps^+\IMP \si$, contradicting the fact that $\ps$ is satisfiable in $\K$. Hence $\m B\models \rho(\abar,\bbar)$, where $\rho(\xbar,\ybar)$ is the conjunction of the negated equations $\neg \si$ for $\si$ ranging over the equations of $\ps^-$ (hence, $\ps=\ps^+\AND\rho$), and so there exists a quantifier-free formula $\chi(\xbar)$ satisfying the conditions for the conservative congruence extension property.

We prove that $\chi$ satisfies (i) and (ii). For (i), consider any algebra $\m C\in \K$ and tuples $\overline{c}\in C^{\xbar}$ and $\overline{d}\in C^{\ybar}$ such that $\m C\models \ps(\overline{c},\overline{d})$. We must prove that $\m C\models \chi(\overline{c})$. Let $\m C'$ be the subalgebra of $\m C$ generated by $\overline{c},\overline{d}$, and note that $\m C'\models \ps(\overline{c},\overline{d})$ as $\ps$ is quantifier-free. Since $\m C'\models \ps^+(\overline{c},\overline{d})$, there is a surjective homomorphism $\m B\onto \m C'$ mapping $\abar$ to $\overline{c}$ and $\bbar$ to $\overline{d}$. Hence, with the obvious interpretation of the new constants, $\m C'\models \Diag(\m B)$. So $\Th(\K)\cup\Diag(\m B)\vdash \rho\IMP \chi$ and $\m C'\models \rho(\overline{c},\overline{d})$ entail $\m C'\models \chi(\overline{c})$, and  therefore $\m C\models \chi(\overline{c})$.

For (ii), consider any algebra $\m A'\in \K$ generated by a tuple $\abar'\in (A')^{\xbar}$ such that $\m A'\models\chi(\abar')$ and, for all equations $\ep(\xbar)$, $\K\models \ps^+\IMP \ep$ entails $\m A'\models\ep(\abar')$. Then there exists a (unique) homomorphism $h\colon \m A\to \m A'$ mapping $\abar$ to $\abar'$. Just observe that, whenever $s(\abar)=t(\abar)$ for two terms $s(\xbar),t(\xbar)$, we have $\m B\models s(\abar)\eq t(\abar)$ and hence also $\K\models \ps^+\IMP s(\xbar)\eq t(\xbar)$, which implies $\m A'\models s(\abar')\eq t(\abar')$. Moreover, $h$ is surjective because $\abar'$ generates $\m A'$. By the conservative congruence extension property, there exist $\m B'\in \K$ extending $\m A'$ and a surjective homomorphism $h'\colon \m B\onto \m B'$ that extends $h$ and satisfies $\m B'\models \rho(\abar',h'(\bbar))$. Using the fact that $\m B\models \ps^+(\abar,\bbar)$ entails $\m B'\models \ps^+(\abar',h'(\bbar))$, we conclude that $\m B'\models \ps(\abar', h'(\bbar))$, as was to be proved.

For the converse direction, let $\m B\in \K$ be an algebra admitting a finite presentation $((\abar,\bbar),\pi(\xbar,\ybar))$, let $\m A$ be the subalgebra of $\m B$ generated by $\abar$, and assume that $\bbar$ does not lie in $\m A$. Further, let $\rho(\xbar,\ybar)$ be a conjunction of negated equations such that $\m B \models \rho(\abar,\bbar)$. Define the conjunction of literals $\ps(\xbar,\ybar)\coloneqq \pi\AND \rho$. Then there exists a quantifier-free formula $\chi(\xbar)$ satisfying the properties in (i) and (ii). Condition~(i) entails easily that
\[
\Th(\K)\cup\Diag(\m B)\vdash \rho(\abar,\bbar)\IMP \chi(\abar).
\] 
Consider next a surjective homomorphism $h\colon \m A\onto \m A'$ with $\m A'\in\K$ and $\m A'\models \chi(h(\abar))$. Note that $h(\abar)$ generates $\m A'$ and, for any equation $\ep(\xbar)$, $\K\models\pi\IMP\ep$ implies $\m B\models \ep(\abar)$ and hence also $\m A\models \ep(\abar)$ and $\m A'\models \ep(h(\abar))$. By~(ii), there exist $\m B'\in \K$ extending $\m A'$ and a tuple $\bbar' \in (B')^{\ybar}$ such that $\m B'\models \ps(h(\abar),\bbar')$. Let $\m B''$ be the subalgebra of $\m B'$ generated by $h(\abar),\bbar'$ and note that $\m B''\in\K$ because $\K$ is a universal class. Since $\m B''\models \pi(h(\abar),\bbar')$, there exists a surjective homomorphism $h'\colon \m B\onto \m B''$ mapping $\abar$ to $h(\abar)$ and $\bbar$ to $\bbar'$. In particular, $h'$ extends $h$. Further, $\m B''$ extends $\m A'$ and satisfies $\m B''\models \rho(h(\abar),\bbar')$. Hence $\K$ has the conservative congruence extension property.
\end{proof}

We are now in a position to compare Wheeler's characterization of universal theories that have a model completion with our results from Section~\ref{s:model-completions}. Following Wheeler, we say that a universal class of algebras with finite presentations $\K$ is \emph{coherent} if, whenever $\m B\in \K$ is finitely presented in $\K$ and $\m A$ is a finitely generated subalgebra of $\m B$, then $\m A$ is finitely presented in $\K$.
It is not difficult to see that a universal class of algebras with finite presentations is coherent if, and only if, it has the variable projection property.

In the case where  $\K$ is a universal class of algebras, the following main result of~\cite{Whe76} is a direct consequence of Theorem~\ref{t:charact-model-completion} and Propositions~\ref{p:strong-CMEP-iff} and~\ref{p:CMEP-CCEP}.

\begin{proposition}[{\cite[Theorem~5]{Whe76}}]
Let $\K$ be a universal class with finite presentations. The theory of $\K$ has a model completion if, and only if, $\K$ is coherent and has the amalgamation property and conservative congruence extension property.
\end{proposition}

%%%%%%%%%%%%%%%%%%%%%%%%%%%%%%%%%%%%%%%%%%

\bibliographystyle{asl}

%%%%%%%%%%%%%%%%%%%%%%%%%%%%%%%%%%%%%%%%%%

\end{document}